\documentclass[12 pt]{amsart}
\usepackage{amscd,amsfonts,amssymb,amsmath}
\usepackage{hyperref}
\usepackage{epsfig}
\usepackage{mathtools}
\usepackage{tabu}
\usepackage{tikz-cd}
\newtheorem{theorem}{Theorem}[section]
\newtheorem{corollary}[theorem]{Corollary}
\newtheorem{lemma}[theorem]{Lemma}
\newtheorem{proposition}[theorem]{Proposition}
\theoremstyle{definition}

\newtheorem{question}[theorem]{Question}

\newtheorem{definition}[theorem]{Definition}
\newtheorem{example}[theorem]{Example}
\newtheorem{remark}[theorem]{Remark}
\newtheorem*{ack}{Acknowledgement}
\usepackage[all,cmtip]{xy}

\usepackage{graphicx} 

\newcommand{\Alex}{\operatorname{Alex}}
\newcommand{\Maps}{\operatorname{Map}}

\newcommand{\Conj}{\operatorname{Conj}}
\newcommand{\Core}{\operatorname{Core}}

\newcommand{\im}{\operatorname{Im}}

\newcommand{\T}{\operatorname{T}}

\newcommand{\R}{\operatorname{R}}

\begin{document}

\author{V.\,G.~Bardakov, D.\,A.~Fedoseev}
\title{
 Invariants of Handlebody-links and Spatial Graphs
}

\begin{abstract}
A {\em $G$-family of quandles} is an algebraic construction which was proposed by A.~Ishii, M.~Iwakiri, Y.~Jang, K.~Oshiro in 2013. The axioms of these algebraic systems were motivated by handlebody-knot theory. In the present work we investigate possible constructions which generalise $G$-family of quandles and other similar constructions (for example, $Q$- and $(G,*,f)$-families of quandles). We provide the necessary conditions under which the resulting object (called an {\em $(X,G,\{*_g\},f,\otimes,\oplus)$-system}) gives a colouring invariant of knotted handlebodies. 

We also discuss several other modifications of the proposed construction, providing invariants of spatial graphs with an arbitrary (finite) set of values of vertex valency. Besides, we consider several examples which in particular showcase the differences between spatial trivalent graph and handlebody-link theories.
\end{abstract}

\maketitle 

\tableofcontents

\section{Introduction}

A spatial graph is a finite graph embedded in the 3-sphere $\mathbb{S}^3$. Two spatial graphs are equivalent if there is an isotopy of $\mathbb{S}^3$ taking one onto the other. S.~Suzuki \cite{S} introduced the notion of the {\em neighborhood equivalence} for spatial graphs. Two spatial graphs are neighborhood equivalent if there is an isotopy of $\mathbb{S}^3$ taking a regular neighborhood of one graph onto that of the other. On the other hand, any handlebody embedded in $\mathbb{S}^3$ is a regular neighborhood of some spatial trivalent graph. Hence, there is a one-to-one correspondence between the set of handlebody-links and that of neighborhood equivalence classes of spatial connected trivalent graphs.

It is known that two spatial trivalent graphs are equivalent if and only if their diagrams are related by a finite sequence of Reidemeister moves. A.~Ishii \cite{I2} introduced moves for spatial trivalent graphs, called $SR$-moves, and showed that two spatial trivalent graphs are neighborhood equivalent if and only if their diagrams are related by a finite sequence of Reidemeister moves and $SR$-moves. Also, in \cite{I2} it was proved  that the cocycle invariant of a spatial trivalent graph is invariant under $SR$-moves. Hence, it is equivalent to defining the invariant for linked handlebodies. In fact, any spatial graph is neighborhood equivalent to some spatial trivalent graph. \newline

In the theory of algebraic systems there exist systems with a set of operations of one type. Let us give some examples. 

{\em Dimonoids} were introduced by J.-L.~Loday \cite{Loday2} in his construction of a universal enveloping algebra for the Leibniz algebra. Dimonoid is a set with two semigroup operations, which are connected by a set of axioms.  Dimonoids are examples of {\em duplexes}, which were introduced by T.~Pirashvili in \cite{Pir}. A duplex is an algebraic system with two associative binary operations (without added relations between these operations). T.~Pirashvili constructed a free duplex generated by a given set via planar trees and proved that the set of all permutations forms a free duplex on an explicitly described set of generators.

In \cite{Kor}, Koreshkov introduced {\em $n$-tuple semigroup} as an algebraic system $\mathcal{S} = (S, *_i, i \in I)$ such that $(S, *_i)$ is a semigroup for any $i \in I$, and such that the following relation holds:
\begin{equation} \label{semigroup_op}
(a *_i b) *_j c = a *_i (b *_j c),\quad a, b, c \in S,\quad i, j \in I. 
\end{equation}

A {\em brace} (skew brace) is a set with two group operations, which satisfy a certain axiom \cite{GV}, \cite{Rump}. A generalization of skew braces was suggested in the paper Bardakov-Neshchadim-Yadav~\cite{BNY}, where brace systems were introduced as a set with a family of group operations which are related by certain axioms.

As a generalization of these examples, in \cite{K} {\em homogeneous algebraic systems} were introduced, as algebraic systems in which all algebraic operations have the same arity and satisfy the same axioms. Particular case of such systems is a semigroup (monoid, group) system $\mathcal{G} = (G, *_i, i \in I)$, where $(G, *_i)$ is a semigroup (monoid, group) for any $i \in I$. An example of a semigroup system with two operations is a duplex. In \cite{K} $\mathcal{G}$ is said to be multi-semigroup (multi-monoid, multi-group) if the operations are related by the condition~(\ref{semigroup_op}).
An example of a multi-semigroup with $n$ operations is an $n$-tuple semigroup \cite{Kor} as described above.

In \cite{BF} quandle systems $\mathcal{Q} = (Q, *_i, i \in I)$ were considered, where $(Q, *_i)$ is a quandle for any $i \in I$, and a multiplication $*_i *_j$ of the operations $*_i$ and $*_j$ is defined by the rule
\begin{equation}
p (*_i *_j) q = (p *_i q) *_j q,\quad p, q \in Q.
\end{equation}

In the general case the algebraic system $(Q, *_i *_j)$ is not a quandle, but if the operations satisfy the axioms
\begin{equation} \label{quand}
(x *_i y) *_j z= (x *_j z) *_i (y *_j z),\quad (x *_j y) *_i z= (x *_i z) *_j (y *_i z),~\quad x, y, z \in Q,
\end{equation}
(that is, the operations are mutually distributive) then $(Q, *_i *_j)$ and  $(Q, *_j *_i)$ are quandles. V.~Turaev \cite{T} called quandle systems which satisfy axioms~(\ref{quand}) for all $i, j \in I$ {\em multi-quandles} and gave them a topological interpretation. \newline




\section{Extensions of quandles} \label{ext}

We begin with recalling some constructions of quandle extensions.

\subsection{$G$-family of quandles and $Q$-family of quandles} \label{sec:g_and_q}

In \cite{I} {\em $G$-family of quandles} was introduced which is an algebraic system whose axioms are motivated by handlebody-knot theory. This construction was further used to construct invariants of handlebody-knots. Let us briefly recall this construction.

\begin{definition} \label{GF}
	Let $G$ be a group. A $G$-family of quandles is a non-empty set $X$ with a family of binary operations $*_g\colon X \times  X \to X$, $g \in  G$, satisfying the following axioms:
\begin{enumerate}
\item  for any $x \in  X$ and any $g \in G$, $x *_g x = x$; 

\item for any $x, y \in X$ and any $g, h \in G$,
$$
x *_{gh} y = (x *_g y) *_h y,~~~x *_e y = x;
$$

\item for any $x, y, z \in X$ and any $g, h \in G$,
$$
(x *_{g} y) *_h z = (x *_h z) *_{h^{-1}gh} (y *_h z).
$$
\end{enumerate}
\end{definition}

In \cite{I} it was proved that for any $g \in G$ the set $X$ with the operation $*_g$ is a quandle $(X, *_g)$, and the direct product $X \times G$ with the operation
\begin{equation} \label{assoc_quandle_op}
(x, g) \cdot (y, h) = (x *_h y, h^{-1} g h)
\end{equation}
is a quandle which is called the {\em associated quandle of $X$}, and is denoted by $(X\times G, \cdot)$.
It is easy to see that the inverse operation $\bar{\cdot} \colon X\times G \to X\times G$ is the following:
\begin{equation} \label{inv_assoc_quandle_op}
(x, g) \bar{\cdot} (y, h) = (x *_{h^{-1}} y, h g h^{-1}).
\end{equation}

In this section we generalise the notion of $G$-family of quandles. Let $G$ be a group. A quandle $Q_G = (G, *)$ on the set $G$ may be defined in many ways. In particular, trivial quandle, $n$-conjugacy quandle,  Takasaki quandle, Alexander quandle, or generalized Alexander quandle are examples of quandles constructed on a given group $G$. Now let $X$ be a non-empty set, and for any $g \in G$ let us have a quandle $(X, *_g)$. A natural question arises:

\begin{question}
What quandles one can define on the direct product $X \times G$?
\end{question}

\begin{example}
1) Let $X = \{ x, y \}$ be a 2-element set, $G = \{ e, g \}$ be a 2-element group in which $g^2 = e$. On $X$ one can define only the structure of trivial quandle. Hence, $(X, *_e)$, $(X, *_g)$ are trivial quandles and we can define the associated quandle as the set $X \times G$ with the operation
$$
(u, a) \cdot (v, b) = (u *_b v, b^{-1} a b).
$$
In our case $(u, a) \cdot (v, b) = (u,  a)$. Hence, we have 4-element trivial quandle.

2) Let $X = \{ x, y, z \}$ be a 3-element set, $G = \{ e, g \}$ be the group from 1). On $X$ one can define three structures of non-isomorphic quandles. Suppose that $(X, *_e)= \T_3$ is the trivial quandle, $(X, *_g) = \R_3$ is the 3-element dihedral quandle. We can define the associated quandle as the set $X \times G$ with the operation
$$
(u, a) \cdot (v, b) = (u *_b v, b^{-1} a b).
$$
In our case $(u, a) \cdot (v, b) = (u *_b v,  a)$. In particular, if $b=e$, then $(u, a) \cdot (v, e) = (u,  a)$. Let us check that we really constructed a quandle on 6 elements. To do it, we can just show that $(X, *_g, g\in G)$ is indeed a $G$-family of quandles, so we need to check axioms from Definition~\ref{GF}. The first axiom is evident. To check the second we shall consider four cases of $*_{eg}, *_{ge}, *_{ee}$, and $*_{gg}$. The first two give us:
\begin{align*}
x *_{g e} y = &x *_g y = (x *_g y) *_e y, \\
x *_{e g} y = &x *_g y = (x *_e y) *_g y.
\end{align*}
The third option yields:
$$x *_{ee} y = x *_e y = x = (x *_e y) *_e y.$$
Finally, the fourth option, taking into account that $*_g$ is the dihedral quandle operation, and the fact that $gg=e$ by the group $G$ definition, yields:
$$x *_{gg} y = x *_e y = x = (x *_g y) *_g y.$$
Finally, to check the third axiom we calculate:
\begin{align*}
(x *_{g} y) *_e z = x*_g y = (x *_e z) *_{g} (y *_e z), \\
(x *_e y) *_g z = x *_g z = (x *_g z) *_e (y *_g z),
\end{align*}
and the cases of $(x *_e y) *_e z$ and $(x *_g y) *_g z$ are evident because both $*_e$ and $*_g$ are quandle operations.
\end{example}

The following construction generalises the definition of $G$-family of quandles.  Let  $G$ be a group with a quandle structure $Q_G = (G, *)$, and let $f\colon G \times G \to G$ be a map. 

\begin{definition}[\cite{BF}] \label{def:Gf-family}
	A $(G, *, f)$-{\it family of quandles} is a non-empty set $X$ with a family of binary operations $*_g\colon X \times  X \to X$, $g \in  G$, satisfying the following axioms:
\begin{enumerate}
\item  for any $x \in  X$ and any $g \in G$, $x *_g x = x$; 

\item for any $x, y \in X$ and any $g, h \in G$,
$$
x *_{gh} y = (x *_g y) *_h y,~~~x *_e y = x,
$$
where $gh$ is the product in $G$ and $e$ is the unit element of $G$;

\item for any $x, y, z \in X$ and any $g, h, q \in G$,
$$
 (x*_{f(g, h)} y) *_{f(g*h, q)} z = (x*_{f(g, q)} z)  *_{f(g*q, h * q)}  (y *_{f(h, q)} z).
$$ 
\end{enumerate}
\end{definition}

Note that unlike the case of $G$-families of quandles here for a given $g \in G$ $(X,*_g)$ is not necessarily a quandle. Nevertheless, the following statement holds:

\begin{proposition}
	Consider a $(G,*,f)$-family of quandles and an arbitrary $g\in G$. Then $*_g$ is idempotent and right-invertible.
\end{proposition}
\begin{proof}
	Idempotency follows directly from the first axiom of $(G,*,f)$-families of quandles. Right invertibility follows from the second axiom, with the inverse operation being $*_{g^{-1}}$ where $g^{-1}$ is the inverse of $g$ in the group $G$.
\end{proof}

A proof of the following lemma can be find in \cite[Lemma 2.4]{BF}.

\begin{lemma}[\cite{BF}] \label{lem:for}
 For any $x, y \in X$ and any $g, h, q \in G$, the function $f$ satisfies the condition
$$
x*_{f(g, h)f(g*h,q)} y = x*_{f(g, q) f(g*q,h*q)} y.
$$ 
\end{lemma}

\begin{example}
\label{ex:f-quandles}
1) If $Q_G = Conj(G)$ is the conjugacy quandle, that is a quandle with the operation $g * h = h^{-1} g h$,  $f(g, h) = h$ for all $g, h \in G$, then by Lemma~\ref{lem:for} we have
$$
x*_{hq} y = x*_{q q^{-1}  h q} y  \Leftrightarrow x*_{hq} y = x*_{ h q} y,
$$ 
and the $(G, *, f)$-family of quandles is the $G$-family of quandles.

2) If $Q_G = T(G)$ is the trivial quandle, that is a quandle with the operation $g * h = g$,  $f(g, h) = h$ for all $g, h \in G$, then by Lemma~\ref{lem:for} we have
$$
x*_{hq} y = x*_{q h} y.
$$ 
Hence, in this case the operations $*_g$ commute in the sense that $(x*_g y)*_h y = (x*_h y)*_g y$ for all $x,y \in X, \, g,h \in G$. For example, that occurs when $G$ is abelian.

3) If $Q_G = Tak(G)$ is the Takasaki quandle, that is a quandle with the operation $g * h = h g^{-1} h$,  $f(g, h) = h$ for all $g, h \in G$, then by Lemma~\ref{lem:for} we have
$$
x*_{hq} y = x*_{q q h^{-1}q} y.
$$ 
It is true, for example, if $G$ has exponent 2.
\end{example}

Suppose that we have a $(G, *, f)$-family of quandles. Define an operation $\cdot$ on the set $X \times G$ by the rule
\begin{equation} \label{gf_family_op}
(x, g) \cdot (y, h) = (x*_{f(g, h)} y, g*h),\quad x, y \in X,  g, h \in G.
\end{equation}

In the paper \cite{BF} the following proposition was announced. Here we give its full proof.

\begin{proposition} \label{prop:adj_f_quandle}
For a $(G, *, f)$-family of quandles the set $X \times G$ with the operation (\ref{gf_family_op}) is a quandle.
\end{proposition}
\begin{proof}
We need to check the quandle axioms.

1) Let $(x, g) \in X \times G$, then $(x, g) \cdot (x, g) = (x *_{f(g, g)} x, g*g)$. Since $*_f(g,g)$ is idempotent and $(G, *)$ is a quandle, we have $(x,g)\cdot (x,g)=(x,g)$, and the first axiom is satisfied.

2) Suppose that $(y, h)$ and $(z, q)$ are two elements of $X \times G$. We need to prove that there is a unique $(x, g) \in X \times G$ such that
$$
(x, g) \cdot (y, h) = (z, q).
$$
It means that
$$
z = x*_{f(g, h)} y, \quad q = g*h,
$$
and $x$ and $g$ must be determined uniquely from these equalities. Since $*_{f(g, h)}$ is right-invertible and $(G, *)$ is a quandle, the second axiom is satisfied.

3) Suppose that $(x, g)$, $(y, h)$ and $(z, q)$ are three  elements of $X \times G$. Then
$$
\left( (x, g) \cdot (y, h) \right) \cdot (z, q) = (x*_{f(g, h)} y, g*h) \cdot (z, q) = \left(  (x*_{f(g, h)} y) *_{f(g*h, q)} z, (g*h) *q \right).
$$
On the other hand,
\begin{multline*}
((x, g) \cdot (z, q)) \cdot ((y, h) \cdot (z, q) )= (x*_{f(g, q)} z, g*q) \cdot (y *_{f(h, q)} z, h * q) = \\
=  \left(  (x*_{f(g, q)} z)  *_{f(g*q, h * q)}  (y *_{f(h, q)} z), (g*q) *(h * q) \right).
\end{multline*}
Since $(G, *)$ is a quandle, to get self-distributivity equality
$$
\left( (x, g) \cdot (y, h) \right) \cdot (z, q) = ((x, g) \cdot (z, q)) \cdot ((y, h) \cdot (z, q) ),
$$
we need to have
$$
 (x*_{f(g, h)} y) *_{f(g*h, q)} z = (x*_{f(g, q)} z)  *_{f(g*q, h * q)}  (y *_{f(h, q)} z).
$$
This equality follows from the third axiom of $(G, *, f)$-family of quandles.
\end{proof}

We will denote the quandle constructed in the proposition by $(X\times G, Q_G, f, \cdot)$ and call it the {\em associated quandle} just as in the case of $G$-families of quandles. It is easy to see that the projection
$$
(X\times G, Q_G, f, \cdot) \to Q_G, \quad (x, g) \mapsto g,
$$
is a quandle homomorphism. 


\begin{remark}
In the definition of $(G, *, f)$-family of quandles we can take a binary operation $\circ = \circ_f \colon G \times G \to G$ instead of the map $f\colon G \times G \to G$.
In this case the axiom (3) in Definition~\ref{def:Gf-family} takes the form

(3')  for any $x, y, z \in X$ and any $g, h, q \in G$,
$$
 (x*_{g \circ  h} y) *_{(g*h) \circ  q} z = (x*_{g \circ  q} z)  *_{(g*q) \circ  (h * q)}  (y *_{h \circ  q} z).
$$ 
Lemma~\ref{lem:for} can then be formulate in the following form.
 For any $x, y \in X$ and any $g, h, q \in G$, the operation $\circ$ satisfies the condition
$$
x*_{(g \circ h)((g*h) \circ q)} y = x*_{(g \circ q) ((g*q) \circ (h*q))} y.
$$ 
Hence, on a group $G$ we have a quandle operation $*$ and some other operation $\circ$.
\end{remark}


\subsection{$Q$-family of quandles and its generalization} \label{sec:q_fam}

A parallel theory may be constructed if instead of working with a group $G$ we work with an arbitrary quandle $Q$. The following definition can be found in
 \cite{I}. 

\begin{definition}[\cite{I}]
	Let $(Q, \circ)$ be a quandle. A $Q$-family of quandles is a non-empty set $X$ with a family of binary operations $*_a\colon X \times  X \to X, a \in  Q$,  satisfying the following axioms.
\begin{enumerate}
\item  for any $x \in  X$ and any $a \in Q$, $x *_a x = x$; 

\item for any $x \in X$ and any $a \in Q$, the map $S_{x, a}\colon X \to X$ defined by $S_{x,a}(y) = y*_a x$
is a bijection;

\item for any $x, y, z \in X$ and any $a, b \in Q$,
$$
(x *_{a} y) *_b z = (x *_b z) *_{a \circ b} (y *_b z).
$$
\end{enumerate}
\end{definition}

Let $(Q, \circ)$ be a quandle, $(X, \{*_a \}_{a\in Q})$ be a $Q$-family of quandles. It can be routinely checked that binary operation 
$$
\cdot \colon (X \times Q) \times (X \times Q) \to X \times Q
$$
defined by the rule
$$
(x, a) \cdot (y, b) = (x *_b y, a \circ b)
$$
gives a quandle structure on the set $X \times Q$. This quandle, just like before, is called the {\em associated quandle} of the family.

We construct a generalization of $Q$-family of quandles. To that end we recall a general construction of quandles which was suggested in~\cite{BarSin}.

\begin{proposition}[\cite{BarSin}] \label{prop:gen_quandle}
Let $X$ and $S$ be two sets, $g\colon X \times X \to \Maps(S \times S, S)$ and $f\colon S \times S \to \Maps(X \times X, X)$ be two maps. Then the set $X \times S$ with the binary operation 
\begin{equation}\label{genralised-quandle-operation}
(x, s)* (y,t)= \big( f_{s, t}(x, y), ~g_{x, y}(s, t) \big)
\end{equation}
forms a quandle if and only if the following conditions hold:
\begin{enumerate}
\item $f_{s, s}(x, x)=x$ and $g_{x, x}(s, s)=s$ for all $x \in X$, $s \in S$;
\item  for each $(y, t) \in X \times S$, the map $(x, s) \mapsto  \big( f_{s, t}(x, y),~g_{x, y}(s, t) \big)$ is a bijection;
\item $f_{g_{x, y}(s, t), u}\Big(f_{s, t}(x, y),~ z \Big)= f_{g_{x, z}(s, u), g_{y, z}(t, u)}\Big(f_{s, u}(x, z), ~f_{t, u}(y, z) \Big)$ 
\item[] and
\item[] $g_{f_{s, t}(x, y), z}\Big(g_{x, y}(s, t), ~u \Big)= g_{f_{s, u}(x, z), f_{t, u}(y, z)}\Big(g_{x, z}(s, u), ~g_{y, z}(t, u) \Big)$.
\end{enumerate}
\end{proposition}

Suppose that $S$ is a quandle $(Q, \circ)$ and $g\colon X \times X \to \Maps(Q \times Q, Q)$ is defined by the rule $g_{x, y}(s, t) = s \circ t$. In this case we get a construction of a quandle.

\begin{proposition}[\cite{A}] \label{prop:spec_quandle}
Let $X$ be a set, $(Q, \circ)$ be a quandle,  and  $f\colon Q \times Q \to \Maps(X \times X, X)$ be a map. Then the set $X \times Q$ with the binary operation 
\begin{equation}\label{genralised-quandle-operation-1}
(x, s)\cdot (y,t)= \big( f_{s, t}(x, y), ~s \circ t \big)
\end{equation}
forms a quandle if and only if the following conditions hold:
\begin{enumerate}
\item $f_{s, s}(x, x)=x$  for all $x \in X$, $s \in Q$;
\item  for each $(y, t) \in X \times Q$, the map $(x, s) \mapsto  \big( f_{s, t}(x, y),~s \circ t \big)$ is a bijection;
\item $f_{s \circ t, u}\Big(f_{s, t}(x, y),~ z \Big)= f_{s \circ u, t \circ u}\Big(f_{s, u}(x, z), ~f_{t, u}(y, z) \Big)$. 
\end{enumerate}
\end{proposition}

We will call the quandle constructed in this proposition the {\em $(Q, \circ, f)$-family of quandles}. Let us show that it generalizes the $Q$-family of quandles. Indeed, suppose that the map $f\colon Q \times Q \to \Maps(X \times X, X)$ depends only on the second argument, in other words 
$$
f_{s,t}(x, y) = f_t(x, y),~~~ s\in Q.
$$
If we put $f_{s,t}(x, y) =: x*_t y$, then it follows from Proposition~\ref{prop:spec_quandle} that:

\begin{corollary}
Let $X$ be a set, 
$(Q, \circ)$ be a quandle,  and  $f\colon Q  \to \Maps(X \times X, X)$ a map, $f_{s,t}(x, y) = x*_t y$. Then the set $X \times Q$ with the binary operation 
\begin{equation}\label{cor}
(x, s)\cdot (y,t)= \big( x *_t y, ~s \circ t \big)
\end{equation}
forms a quandle if and only if the following conditions hold:
\begin{enumerate}
\item $x *_t x=x$  for all $x \in X$, $t \in Q$;
\item  for each $(y, t) \in X \times Q$, the map $(x, s) \mapsto  \big( x*_t y,~s \circ t \big)$ is a bijection;
\item $(x*_t y) *_u z =  (x *_u z) *_{t \circ u} (y *_u z)$. 
\end{enumerate}
\end{corollary}

\begin{remark}
We see that the second axioms in the definitions of $G$-family of quandles and $Q$-family of quandles are different. The $G$-family is stronger: it implies that the mapping $S_{x,a}(y)=y*_a x$ is a bijection, but it also states what is the inverse operation: it is given by $S_{x,a^{-1}}$. In particular, $S_{x,e}$ is the trivial operation. Hence, the trivial element of the group $e\in G$ corresponds to the trivial quandle $(X,*_e)$.

For example, if the group $G$ consists of one element, the $G$-family of quandles also consists of one element, which has to be the trivial quandle. In case of $Q$-families this does not hold: for a 1-element quandle $Q$ the corresponding $Q$-family may consist of {\em one arbitrary quandle}. 

Later we will see (Section~\ref{sec:handlebody}) that the $G$-family axiom is more important in handlebody-knot theory.
\end{remark}

\begin{question}
Is it possible to give a general construction which includes $(G, *, f)$-family of quandles and $(Q, \circ, f)$-family of quandles?
\end{question}

\bigskip


\section{$(G,*,f)$-families of quandles and handlebody-links} \label{sec:ab}
\label{sec:handlebody}

\subsection{Handlebody-links}
We recall the definition of {\em handlebody-links} following \cite{I2}.

\begin{definition}
\label{def:handlebody_link}
	A {\em handlebody-link} is a disjoint union of handlebodies embedded in $\mathbb{R}^3$.  A {\em handlebody-knot} is a one
component handlebody-link. Two handlebody-links  are {\em equivalent} if there is an orientation-preserving self-homeomorphism of $\mathbb{R}^3$  which sends one to the other. 
\end{definition}

 A {\em spatial graph} is a finite graph embedded in  $\mathbb{R}^3$. Two
spatial graphs are {\em equivalent} if there is an orientation-preserving self-homeomorphism of  $\mathbb{R}^3$ which sends one
to the other.

Each handlebody-link may be considered as a regular neighbourhood of a spatial trivalent graph (we allow the graph to have circular components without vertices). This graph is said to {\em represent} the handlebody-link. A {\em diagram} of a handlebody link is the diagram of a spatial graph representing it, that is a general position projection of the graph onto a plane such that at every double point of the projection the structure of over- and undercrossing is supplied. Naturally, different diagrams may correspond to the same handlebody-link. It is known (see \cite{I2}), that two diagrams represent equivalent handlebody-links if and only if those diagrams may be connected by a sequence of moves depicted in Fig.~\ref{fig:reid_handle} (to be precise, the moves $R_1, R_2, TR_1$, and $TR_2$ should be generalised in the following manner: for each move $M\in \{R_1,R_2,TR_1,TR_2\}$ also consider the move $M'$ obtained by changing all undercrossings to overcrossings and vice-versa).

\begin{figure}
\centering\includegraphics[width=300pt]{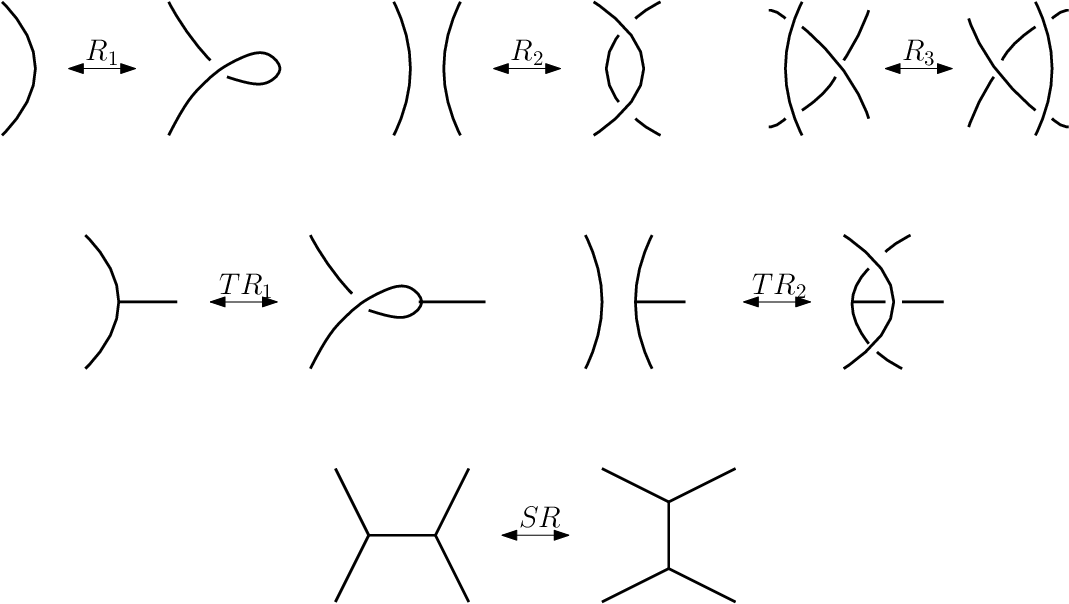}
\caption{Moves on handlebody-link diagrams}\label{fig:reid_handle}
\end{figure}

The names of the moves are chosen as follows: the moves $R_1-R_3$ are the usual Reidemeister moves; moves $TR_1, TR_2$ are the first and the second  Reidemeister moves with an additional arc producing a trivalent vertex; finally, the $SR$ move is analogous to ``saddle move'' or ``handle move'' known in many topological theories. In literature the last move is also known as {\em $HI$-move}, because the shape of the transforming graphs resemble letters H and I.

\begin{remark}
We shall note that the theory of handlebody-links is different from the theory of  spatial graphs. It was remarked by S.\,V.~Matveev and A.\,T.~Fomenko (see Figures 306-307 in the book~\cite{MF}) and illustrated with the following example. The ``topological man'' with linked fingers (see Figure~\ref{fig:topol_man}a, top) can unlink his fingers in the category of handlenody-links (see Figure~\ref{fig:finger_unlink}), but he cannot unlink fingers in the category of spatial graphs, see Section~\ref{sec:graphs} below.
\end{remark}

\begin{figure}
\centering\includegraphics[width=250pt]{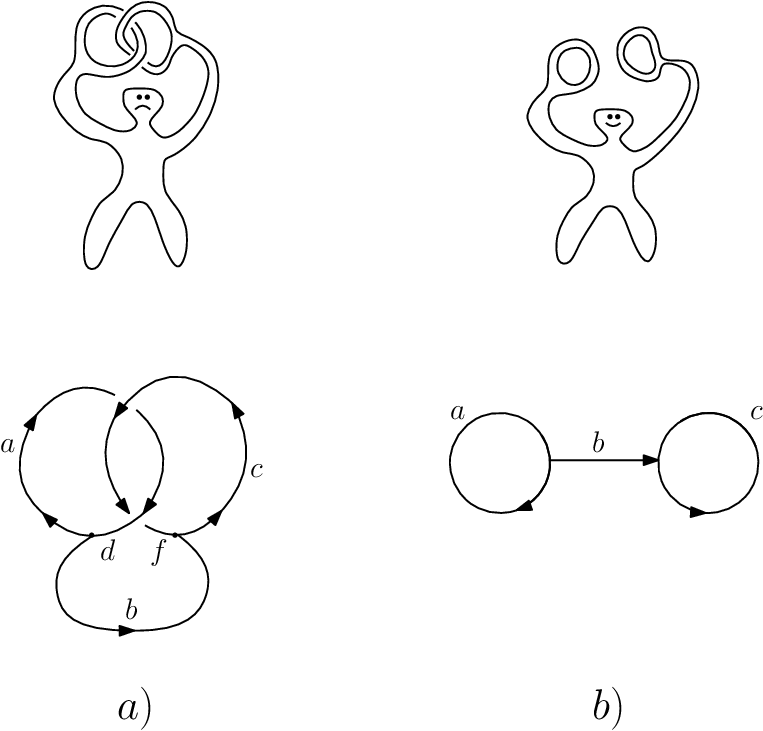}
\caption{``Topological man'' with linked (a) and unlinked (b) fingers, and its graph representation}\label{fig:topol_man}
\end{figure}

\begin{figure} 
\centering\includegraphics[width=250pt]{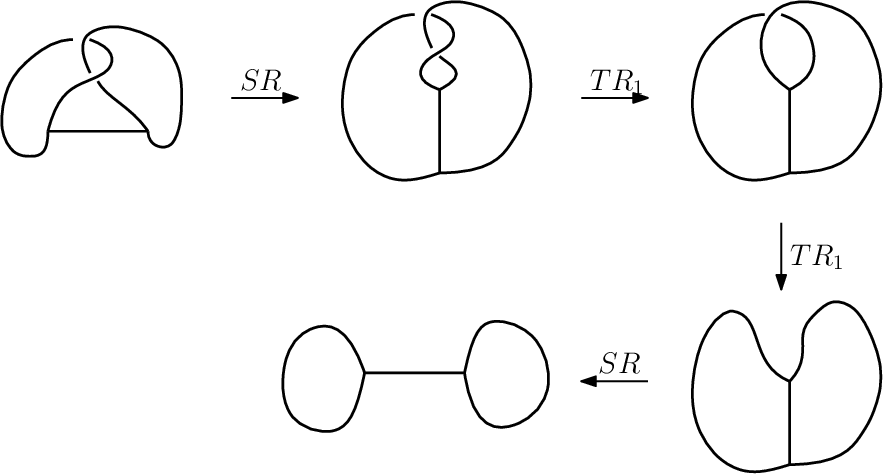}
\caption{How to unlink the fingers of the topological man in the category of handlebody-links. First figure depicts the graph-representation of the man with linked fingers, while the last is the representation of the man with unlinked fingers}\label{fig:finger_unlink}
\end{figure}

Given a $(G,*,f)$-family of quandles we define a {\em proper colouring} of a handlebody-link diagram by the associated quandle $X\times G$ in the following manner. First, let us orient the edges of the trivalent graph arbitrarily. Hence we obtain an oriented diagram of the handlebody-link. A {\em colouring} of the diagram $D$ by the associated quandle $X\times G$ is a map $c\colon\mathcal{A}(D)\to X\times G$, where $\mathcal{A}(D)$ stands for the set of {\em arcs} of the diagram, that is curves whose ends are either undercrossings or trivalent vertices. Each colour here is a pair $f=(x,g)\in X\times G$. We shall call $x$ the {\em $X$-element} of $f$ and $g$ its {\em $G$-element}.

\begin{definition}
	A colouring $c$ of a diagram $D$ is called {\em proper} if it satisfies the following conditions:
	
	\begin{enumerate}
	\item at each crossing the {\em quandle colouring condition} is satisfied as shown in Fig.~\ref{fig:color_rules}$a)$, where $(x,g),(y,h) \in X\times G$, and $\cdot$ is the associated quandle operation defined by the formula~(\ref{gf_family_op});
	\item for each trivalent vertex the colouring rule is shown in Fig.~\ref{fig:color_rules}$b)$; if any of the orientations of the arcs differ from the ones shown, the corresponding element of the group $G$ should be changed to its inverse.
	\end{enumerate}
\end{definition}
 
\begin{remark}
1) If we take $X = \{x\}$ to be the one element  quandle, then it is easy to see that $X\times G\simeq G$ is a homomorphic image of the fundamental group of the complement of the spatial graph in $\mathbb{R}^3$, which is defined by the diagram $D$ if we put $g * h = h^{-1} g h$.

2) If the diagram $D$ does not contain trivalent vertices, we can take $G = \{e\}$ is the trivial group, then the quandle $(X\times G, \cdot)$ is a homomorphic image of the fundamental quandle of the link which corresponds to the diagram $D$.
 \end{remark}

This construction was originally proposed by Ishii {\em et al.} in \cite{I} for the case of $G$-families of quandles (in this case $f(g, h) = h$ for all $g, h \in G$). It was proved that the following theorem holds:

\begin{theorem}
	\label{thm:color_invariant}
	Let $D$ be an oriented diagram of handlebody-link, and the diagram $D'$ be obtained from $D$ by performing one of the moves described above with orientation which agrees with orientation of $D$ on the corresponding arcs. Then the numbers of proper colourings of $D$ and $D'$ by the associated quandle of a finite $G$-family of finite quandles are equal.
\end{theorem}


\begin{figure}
\centering\includegraphics[width=250pt]{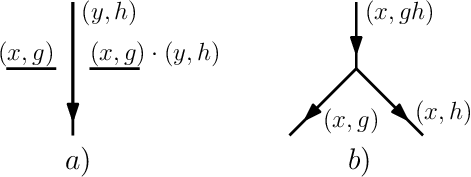}
\caption{Colouring rules for handlebody-links diagrams by a $G$-family of quandles}\label{fig:color_rules}
\end{figure}

The agreement of orientations means the following. For the moves $R_1-R_3, TR_1, TR_2$ there is a natural bijection of the edges of the diagrams before and after the move. Orientations of the corresponding edges must be the same. For the move $SR$ the natural bijection fails: the small horizontal edge in the left-hand side of the move does not correspond to any edge in the right-hand side, and likewise  the vertical edge in the right-hand side does not have a counterpart. In that case we ask all other edges in two sides of the move to have compatible orientations, and these two ``disappearing'' and ``emerging'' edges may be oriented arbitrarily.

This theorem actually gives an invariant of (unoriented) {\em handlebody-links} if we can find a mapping $\rho\colon G\to G$ which is a {\em good involution} (see Definition~\ref{def:good_inv} below). That is the case, because the mapping $\rho(g)\colon g\to g^{-1}$ is indeed a good involution. 

It would be natural to suppose that for $(G,*,f)$-families of quandles some analogous construction should give roughly the same theorem. Surprisingly, it is not the case. It is true, that $G$-family of quandles is a special case of $(G,*,f)$-family with $f(g,h)=h$. The problem is, that for non-trivial functions $f$ it does not seem possible to formulate reasonable conditions under which the theorem holds.

First, let us mention where the general $(G,*,f)$-family approach does {\em not} fail. The moves $R_1-R_3$ do not pose a problem: $(X\times G,\cdot)$  is a quandle and quandle colouring is invariant under Reidemeister moves. Likewise, there is no problem with the $TR_1$ move: we can note that due to the colouring rule of trivalent vertices the $X$-elements of colours of three arcs incident to a trivalent vertex coincide. The required invariance then follows from the fact that $x*_g x = x$ for any $g\in G$ in $(G,*,f)$-family, and an argument parallel to the usual proof of quandle colouring invariance under the first Reidemeister move. Finally, the $SR$ move does not break the invariance because it only works with the $G$-element colouring conditions, and the rules are the same for any $(G,*,f)$-family, including $G$-families.

The obstacle appears when studying the $TR_2$ move. As we can see, the parity of the number of crossings in the lefthand and the righthand sides of this move is different. Hence, before and after the move the operation $f$ would be applied the number of times of different parities as well. Therefore, if $f$ is non-trivial, the invariance may appear only under some very special and complicated conditions.

In a sense $TR_2$ is the most ``interesting'' move in the handlebody-link theory. To overcome this difficulty we have to work in a broader set of objects than just $(G,*,f)$-families. Then, imposing certain conditions (which $G$-families satisfy), we get an invariance theorem.

\subsection{$(X,G,\{*_g\},f,\otimes,\oplus)$-systems} Lets us now generalise the notion of $(G,*,f)-$families of quandles. Then we will impose additional conditionas on it to obtain an invariant of handlebody-links.

We shall go step by step defining the necessary algebraic structures. We begin with the definition of an {\em $(f,\otimes)$-system}. Consider sets $X$ and $G$, a map $f\colon G\times G\to G$, a binary operation $\otimes\colon G\times G\to G$, and a family of binary operations $*_g\colon X\times X \to X$ for each $g\in G$. 

\begin{definition}
	\label{def:ab-system}
	The triple $(X,G,*_g)$ is called {\em an $(f,\otimes)$-system} if the following conditions are satisfied:
	\begin{enumerate}
		\item for all $x\in X$ and $g\in G$, $x*_{f(g,g)} x = x$;
		\item for any $g,h\in G$ and any given $y,z\in X$,  there exists a unique $x\in X$ such that $x*_{f(g,h)} y=z$;
		\item for any $x,y,z\in X$ and $g,h,q\in G$, $$(x *_{f(g, h)} y) *_{f(g \otimes h, q)} z = (x*_{f(g, q)} z)  *_{f(g \otimes q, h \otimes q)}  (y *_{f(h, q)} z).$$
	\end{enumerate}
\end{definition}

Let us call the set $X\times G$ the {\em associated set} of the system, and let the operation $\cdot\colon (X\times G) \times (X\times G) \to X\times G$ be defined as $$(x,g)\cdot(y,h)=(x*_{f(g,h)} y,g \otimes h).$$

\begin{theorem}
	If $(G,\otimes)$ is a quandle, then $(X\times G ,\cdot)$ is a quandle.
\end{theorem}

\begin{proof}
	Let us consequently check all quandle axioms.
	
	1. {\bf Idempotency.} We need to check the conditions under which $(x,g)\cdot (x,g) = (x,g)$ for all $x\in X, g\in G$. By definition, $$(x,g)\cdot (x,g) = (x*_{f(g,g)} x, \beta(g,g)).$$ By axiom (1) from Definition~\ref{def:ab-system} $x *_{f(g,g)} x = x$ for all $x\in X, g\in G$. Therefore it is necessary and sufficient that $g\otimes g=g$ for all $g\in G$, in other words, the operation $\otimes$ must be idempotent.
	
	2. {\bf Right invertibility.} Consider the equality $(z,q) = (x,g)\cdot (y,h)$ as an equation on $(x,g)$. Writing it explicitly, we get $$(z,q) = (x *_{f(g,h)} y, g \otimes h)$$ which translates into two equations: $z = x *_{f(g,h)} y$ and $q = g \otimes h$. For these equations to have a single solution $(x,g)$ we need the operation $*_{f(g,h)}$ to be right invertible for all $g,h \in G$, and the operation $\otimes$ to be right invertible as well.
	
	3. {\bf Self-distributivity.} We shall write the sequence of equalities explicitly: 
	\begin{multline*}
		((x,g)\cdot (y,h))\cdot (z,q) = (x *_{f(g,h)} y, g \otimes h) \cdot (z,q) = \\ = ((x *_{f(g,h)} y) *_{f(g \otimes h,q)} z, (g \otimes h)\otimes q) = ((x*_{f(g, q)} z)  *_{f(g \otimes q, h \otimes q)}  (y *_{f(h, q)} z), (g \otimes h) \otimes q)
	\end{multline*}
	by axiom (3) from Definition~\ref{def:ab-system}. On the other hand,
	\begin{multline*}
		((x,g)\cdot(z,q))\cdot ((y,h)\cdot(z,q)) = (x*_{f(g,q)} z, g \otimes q)\cdot (y*_{f(h,q)} z, h \otimes q) = \\ ((x*_{f(g, q)} z)  *_{f(g \otimes q, h \otimes q)}  (y *_{f(h, q)} z), (g \otimes q)\otimes (h \otimes q)).
	\end{multline*}
	Therefore, for these expressions to be equal, it is necessary and sufficient that $(g \otimes h) \otimes q = (g \otimes q)\otimes (h \otimes q)$, in other words, for $\otimes$ to be self-distributive.
\end{proof}

A simple check verifies the following

\begin{proposition}
	$G$-family of quandles and $Q$-family of quandles are $(f,\otimes)$-systems with their associated sets being quandles. 
\end{proposition}

From now on we will consider this condition to be satisfied (since our goal is to construct colouring invariants of graphs by $(X\times G,\cdot)$). This way we defined a rich family of ``enhanced'' quandles $(X\times G,\cdot)$ for any quandle $(G,\otimes)$, giving colouring invariants of links. To obtain invariants of handlebody-links and spatial graphs, we need more instruments yet.

\begin{definition}
	\label{def:good_inv}
	Let $(Q,*)$ be a quandle. A mapping $\rho\colon Q\to Q$ is called a {\em good involution} if the following conditions are satisfied:
	\begin{enumerate}
		\item $\rho^2=id$,
		\item for any $u,v\in Q$, $(u*v)*\rho(v)=u$;
		\item for any $u,v\in Q$, $\rho(u)*v=\rho(u*v)$.
	\end{enumerate}
\end{definition}

\begin{remark}
In literature (see, for example, \cite{I2}) one can find another definition of good involution. There instead of the condition (2) one has the axiom (2') $u * \rho(v) = u \bar{*} v$, where $\bar{*}$ denotes the inverse of $*$ (that is, if one has equation $z = x * y$ which has a unique $x$ as a solution, then $x = z \bar{*} y$, and that defines the operation $\bar{*}$). These definitions are equivalent. Indeed, let us begin with the axiom (2'): for any $u,v\in Q$, $u * \rho(v) = u \bar{*} v$. It is equivalent to $(u * v) * (\rho(v) * v) = u$. Due to axiom (3) (replacing $u$ with $v$) we have $\rho(v) * v = \rho (v * v) = \rho(v)$, therefore $u = (u * v) * (\rho(v) * v) = (u * v) * \rho(v)$, and hence we get axiom (2). All transitions here were equivalences, therefore the defintions are equivalent.
\end{remark}

The notion of good involution is important in colouring invariants theory for the following reason. Consider a quandle $(Q,*)$, a good involution $\rho$,  and a link diagram $D$. Let us fix some orientation of the diagram $D$ and denote by $D'$ the diagram obtained by inverting the orientation. For each proper colouring $c(D)$ define the colouring $c(D')$ by replacing each colour $g\in Q$ of the arcs of the diagram with $\rho(g)$. It is easy to check that the number of proper colourings $c(D)$ equals the number of proper colourings $c(D')$. Hence, the quandle defines a colouring invariant of an {\em unoriented} link.

Without going into detail let us also mention that exactly the fact that the operation $\rho\colon (x,g)\mapsto (x,g^{-1})$ is a good involution in case of $G$-families of quandles makes the Ishii colouring invariant an invariant of handlebody-links. 

\begin{definition}
	An $(f,\otimes)$-system is called an {\em $(f,\otimes)$-system with good involution} if it is equipped with a mapping $\rho\colon X\times G\to X\times G$ which is a good involution on the associated quandle $(X\times G,\cdot)$. 
\end{definition}

Given an involution $\rho$ on the associated set $X\times G$ we shall write $\rho(x,g)=(\rho_g(x),\rho_x(g))$ bearing in mind that $\rho(x)$ depends on $g$ and $\rho(g)$ depends on $x$.

Now let us introduce one more binary operation $\oplus\colon G\to G$.

\begin{definition}
	\label{abg-system+}
	An $(f,\otimes)$-system with the binary operation $\oplus$ is called an {\em $(X,G,\{*_g\},f,\otimes,\oplus)$-system} if for all $x,y\in X, g,h\in G$, $$(x*_g y)*_h y=x*_{g \oplus h} y.$$
\end{definition}

The construction of $(X,G,\{*_g\},f,\otimes,\oplus)$-system is very general. For example, if $G$ is a group, we set $g \otimes h=g*h, \,g \oplus h=gh$, strengthen the first condition to $x*_g x=x$ for all $g\in G$, and add the condition $x*_e y=x$, then we get exactly the $(G,*,f)$-family of quandles. Furthermore, let $f(g,h)=h, \,g \otimes h=g, \,g \oplus h=gh$. If all operations $*_g$ are right invertible, we get a family of quandles closed under quandle multiplication as described in~\cite{BF}. Elementwise, we can describe it in the associated set language. If we set $(x,g)\circ (y,h) := (x*_{gh} y, gh)$, this operation adequately describes the quandle multiplication: to multiply $x\in (X,*_g)$ by $y\in (X,*_h)$ we first consider them both as elements of $(X,*_{gh})$ and then multiply inside this quandle.

To obtain an invariant of trivalent graphs (which are encoded by planar diagrams modulo moves $R_1-R_3, TR_1, TR_2$) we need to impose some additional conditions on a $(X,G,\{*_g\},f,\otimes,\oplus)$-system.
	
\begin{definition}	\label{def:tr-comp}
	Consider a $(X,G,\{*_g\},f,\otimes,\oplus)$-system and an involution $\rho$ on $G$. Let us extend it on the associated quandle by the rule $\rho(x,g)=(x,\rho_x(g))$ for all $x\in x, g\in G$. This system is called {\em trivalent-compatible} if the following conditions are satisfied:
	\begin{enumerate}
		\item for all $g,h\in G$, $h\oplus (g\otimes h) = g\oplus h$;
		\item the mapping $f$ depends only on its second argument: $f(g,h)=:f(h)$;
		\item for all $g,h,q\in G$, $g\otimes (h \oplus q)=(g \otimes h)\otimes q$ (twisted associativity); 
		\item for all $g,h,q\in G$, $f(h \oplus q)=f(h)\oplus f(q)$ (``additivity'' of $f$);
		\item for all $g,u,v\in G$, $(u \oplus v)\otimes g=(u \otimes g)\oplus (v \otimes g)$ (distributivity);
		\item for all $g,h\in G$, $h\oplus \rho_x(g \oplus h)=\rho_x(g), \, \rho_x(g \oplus h)\oplus g=\rho_x(h)$ for all $x\in X$.
	\end{enumerate}
\end{definition}

\begin{remark}
	Note that the definition of trivalent-compatible system does not require the involution $\rho$ to be good; it is stated for {\em any} involution $\rho$, though the involution must be good in order to construct invariants of trivalent graphs or handlebody-links. For an involution to be good on the associated quandle, it should be good on the quandle $G$, and in addition for all $x,y \in X, g\in G$ the following condition must hold: $(x*_{f(g)} y)*_{\rho(f(g))} y = x$. 
\end{remark}

Consider a diagram of a spatial trivalent graph and an associated quandle of a $(X,G,\{*_g\},f,\otimes,\oplus)$-system. We define a {\em proper colouring} of the diagram in the following way. First, we orient all edges of the diagram arbitrarily. Then colour each edge of the (oriented) diagram by the elements of the associated quandle $(X,G)$. The colouring is proper if at every crossing and trivalent vertex the conditions shown in Fig.~\ref{fig:color_rules_general} are satisfied. Edge orientation reversal corresponds to replacing the colour with its image under the involution $\rho$, that is $(x,g)$ goes to $\rho(x,g)$. In particualr, if at a trivalent vertex orientations of the edges do not match the ones on Fig~\ref{fig:color_rules_general}b), one should reverse the orientations and change the colours appropriately, then check if the colouring is proper. 

\begin{figure}
\centering\includegraphics[width=250pt]{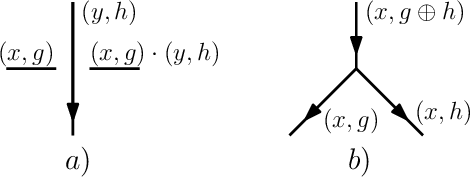}
\caption{Colouring rules for handlebody-links diagrams in the general case}\label{fig:color_rules_general}
\end{figure}

The following theorem holds by construction:

\begin{theorem} \label{thm:tri-graph-inv}
	Consider a trivalent-compatible $(X,G,\{*_g\},f,\otimes,\oplus)$-system. Let $(X\times G,\cdot)$ be a quandle and let $\rho\colon (x,g)\mapsto (x,\rho(g))$ be a good involution. Then the number of proper colourings with the associated quandle $(X\times G,\cdot)$ is an invariant of spatial trivalent graphs.
\end{theorem}

\begin{proof}
	The conditions from Definition~\ref{def:tr-comp} were obtained by considering each move on spatial trivalent graph diagrams and establishing the conditions which ensure bijection between the sets of proper colourings of the diagrams before and after the move. Let us note that due to $\rho$ being good involution, proper colouring definition for trivalent vertices, and condition 6 (see discussion below) we can orient arcs of the diagram arbitrarily.
	
	The number of proper colourings is naturally invariant under the ordinary Reidemeister moves $R_1, R_2, R_3$ because the colouring is done with a quandle. We now need to consider the moves appearing in the spatial trivalent-graph theory and one additional vertex rotation operation.
	
\begin{figure}
\centering\includegraphics[width=250pt]{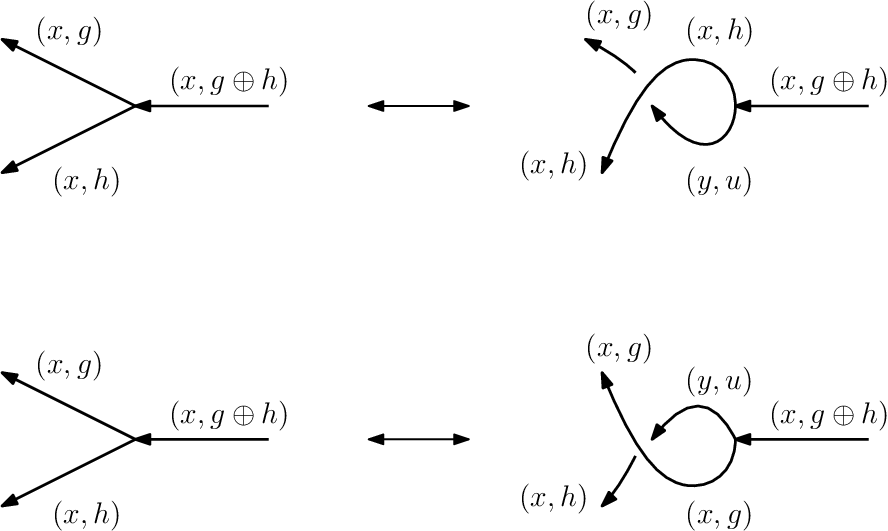}
\caption{The move $TR_1$}\label{fig:TR1}
\end{figure}

\begin{enumerate}
	\item We begin with the move $TR_1$, see Fig.~\ref{fig:TR1}. Let us consider one type of the move, see the top row of the figure. We get $$(y,u) = (x, g\otimes h),$$ $$h\oplus u = g\oplus h,$$ therefore $h\oplus (g\otimes h) = g\oplus h,$ which is condition 1. The other type of the move gives the same condition.
	\begin{remark}
		Note that here (as well as in the consideration of other moves below) we actually only need to establish conditions, which guarantee compatibility of all relations appearing from the move. That stems from the fact, that in that case all ``new'' colours appearing during the move can be correctly defined by these relations, and hence the number of proper colouring will be preserved by the move.
	\end{remark}
	\item Now let us consider the second move, Fig.~\ref{fig:TR2}. From the first type of the move we get:
	\[
	\begin{cases}
	(z_1,w_1)\cdot (x,g) = (y,u) \\
	(z_2,w_2)\cdot (x,g) = (y,v) \\
	(z_3,w_3)\cdot (x,g) = (y,u\oplus v) \\
	z_1=z_2=z_3=:z \\
	w_3 = w_1\oplus w_2
	\end{cases}
	\quad \Rightarrow \quad
	\begin{cases}
	(z,w_1)\cdot (x,g) = (y,u) \\
	(z,w_2)\cdot (x,g) = (y,v) \\
	(z,w_1\oplus w_2)\cdot (x,g) = (y,u\oplus v).
	\end{cases}
	\]
Since these are equations on the elements of the associated quandle, each is broken into two (for the first and the second element of the pair). First elements give us the relation $$y=z*_{f(w_1,g)} x = z*_{f(w_2,g))} x = z*_{f(w_1\oplus w_2,g)} x$$ which is guaranteed by condition 2. Second elements give us
	\[
	\begin{cases}
		w_1\otimes g = u \\
		w_2 \otimes g = v \\
		(w_1\oplus w_2) \otimes g = u\oplus v
	\end{cases}
	\quad \Rightarrow \quad
	(w_1 \oplus w_2)\otimes g=(w_1 \otimes g)\oplus (w_2 \otimes g),
	\]
	and that is condition 5 up to notation change.
	
	Interestingly enough, this time the second type of the move (Fig.~\ref{fig:TR2}, bottom row) gives us different relations. To be precise, we get
	\[
	\begin{cases}
		(z_1,w_1)\cdot (y,u) = (x,g) \\
		(x,g)\cdot (y,v) = (z_2,w_2) \\
		(z_1,w_1)\cdot (y,u\oplus v) = (z_2, w_2)
	\end{cases}
	\quad \Leftrightarrow \quad
	\begin{cases}
		z_1 *_{f(w_1,u)} y = x \\
		z_1 *_{f(w_1,u\oplus v)} y = z_2 \\
		x *_{f(g,v)} y = z_2 \\
		w_1 \otimes u = g \\
		w_1 \otimes (u\oplus v) = w_2 \\
		g\otimes v = w_2.
	\end{cases}
	\]
	From the second triple of equations we immediately get $w_1 \otimes (u\oplus v) = (w_1 \otimes u) \otimes v$, which is condition 3. From the first triple we get $(z_1 *_{f(w_1,u)} u) *_{f(g,v)} y = z_1 *_{f(w_1, u\oplus v)} y$. Transforming the left-hand side of this equality by definition of the system we get $$z_1 *_{f(w_1, u)\oplus f(g, v)} y = z_1 *_{f(w_1, u\oplus v)} y.$$ Taking into account that $w_1 \otimes u = g$, we get $$*_{f(w_1,u)\oplus f(w_1\otimes u, v)} = *_{f(w_1, u\oplus v)}$$ which is guaranteed by condition 4 together with condition 2. 
	
	\item Finally we shall discuss condition 6. As we said before, one can reverse orientations of edges together with applying the involution $\rho$ to the corresponding colour. Therefore, any edge of a trivalent vertex can be considering ``incoming'', and other two may be considered ``outcoming''. Hence we need the compatibility condition 6 to cover that freedom. Fig.~\ref{fig:trivalent_rotate} shows this procedure. We take a properly coloured trivalent vertex, then rotate it one of two ways (depicted in the top and bottom rows of the figure), and then reorient vertices together with changing the colours. From the rightmost figure of the top row we get the condition $h\oplus \rho_x(g \oplus h)=\rho_x(g)$, and from the rightmost figure of the bottom row we get condition $\rho_x(g \oplus h)\oplus g=\rho_x(h)$. These two conditions together constitute condition 6 of the theorem.

\begin{figure}
\centering\includegraphics[width=250pt]{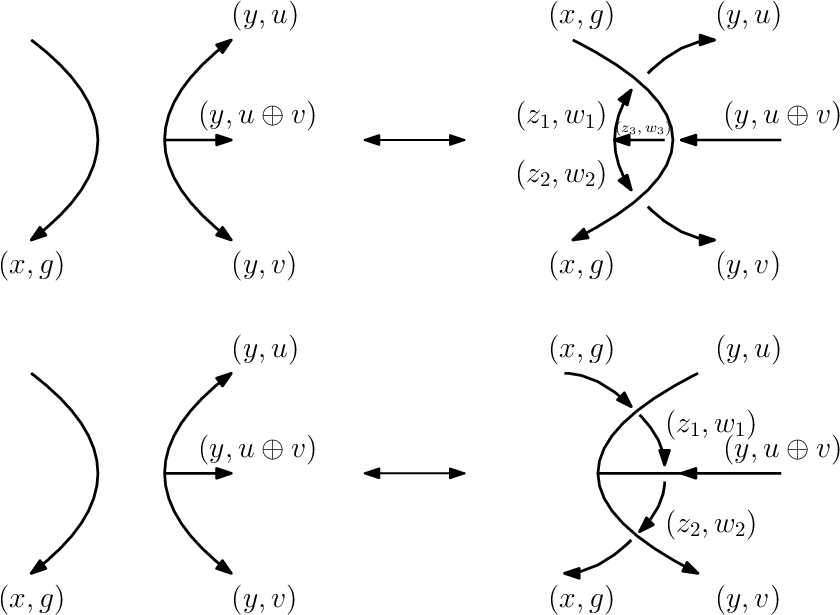}
\caption{The move $TR_2$}\label{fig:TR2}
\end{figure}

\end{enumerate}


\begin{figure}
\centering\includegraphics[width=300pt]{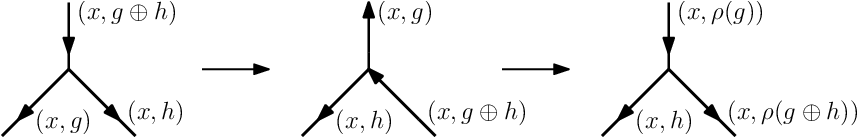}
\caption{Rotation of a trivalent vertex and the corresponding orientation and colour change}\label{fig:trivalent_rotate}
\end{figure}

We considered all moves on spatial trivalent graphs. Together with compatibility condition 6 the obtained conditions guarantee that equations appearing on two sides of each move are compatibale and uniquely determine all emerging colours. Therefore, the number of proper colourings remains unchanged by the moves, and hence is an invariant. \end{proof}

The final upgrade of our algebraic structure goes as follows:

\begin{definition} \label{def:assoc}
	An $(X,G,\{*_g\},f,\otimes,\oplus)$-system is said to have {\em associative composition function} if for all $g,h,q\in G$, $$g\oplus (h \oplus q)=(g \oplus h)\oplus q.$$
\end{definition}

The associativity condition is needed exactly for the number of colourings to be invariant under the move $SR$, therefore we get the following theorem:

\begin{theorem}
	Consider a trivalent-compatible $(X,G,\{*_g\},f,\otimes,\oplus)$-system with associative composition function. Let $\rho$ be good involution and the associated set be a quandle. Then the number of proper colourings with the associated quandle (defined in the same way as for spatial trivalent graphs) is an invariant of handlebody-links.
\end{theorem}

\begin{proof}
We need to consider only $SR$ move since all other moves were already covered in the previous theorem; it is depicted in Fig.~\ref{fig:SR}. From the left-hand side we deduce that
	\[
	\begin{cases}
		g = \rho(l)\oplus \rho(h) \\
		\rho(l) = u\oplus v 
	\end{cases}
	\quad \Rightarrow \quad
	g = (u\oplus v)\oplus \rho(h).
	\]
	Likewise, from the right-hand side we get
	\[
	\begin{cases}
		g = u\oplus w \\
		w= v\oplus \rho(h) 
	\end{cases}
	\quad \Rightarrow \quad
	g = u\oplus (v\oplus \rho(h)).
	\]	
	Therefore, $(u\oplus v)\oplus \rho(h) = u\oplus (v\oplus \rho(h))$ for all $u,v,h\in G$, that is, $\oplus$ operation must be associative.
\end{proof}

\begin{figure}
\centering\includegraphics[width=250pt]{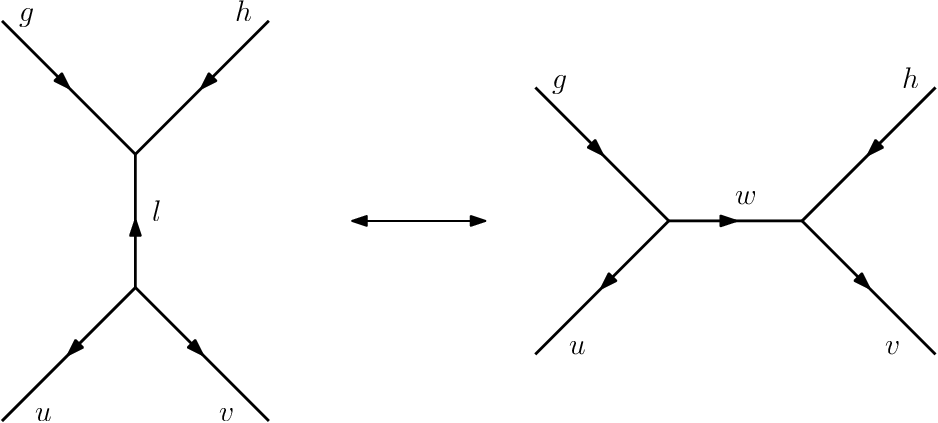}
\caption{The move $SR$}\label{fig:SR}
\end{figure}

\begin{remark}
	We shall remark that multiple trivial quandle, multiple Takasaki quandle, multiple Core quandle, multiple Alexander quandle and other multiple quandles are {\em not} $(X,G,\{*_g\},f,\otimes,\oplus)$-systems. Nevertheless, they can be though of in the similar manner if the composition operation $\oplus$ would be allowed to be defined not on all $G\times G$ but on some subset thereof.
\end{remark}

\begin{example}
	It turns out that {\em axets} (see~\cite{InroyShpec}) give an example of (trivalent-compatible) $(X,G,\{*_g\},f,\otimes,\oplus)$-systems. Recall that if $S$ is a group, an {\em $S$-axet} is a triple $(X,G,\tau)$, where $X$ is a set, $G$ is a group acting on $X$, and $\tau\colon X\times S\to G$, $\tau\colon (x,s)\mapsto \tau_x(s)$, such that
	\begin{enumerate}
		\item for any $(x,s)\in X\times S$,  $\tau_x(s)\in G_x$ where $G_x$ is the stabiliser of $x$;
		\item for any $x \in X$, $s, s'\in S$, $\tau_x(s)\tau_x(s')=\tau_x(ss')$;
		\item  for any $ x\in X$, $s\in S, g\in G$, $\tau_{gx}(s)=g\tau_x(s)g^{-1}$, where $gx$ is the image of the action of $g$ on $x$.
	\end{enumerate}
\end{example}

\begin{proposition}
	Let $S$ be a group and $(X,G,\tau)$ be an $S$-axet. Define $f,\beta\colon S\times S\to S$
	$$f(s,s')=s', \;\;\; s\otimes s'=s;$$
	and $\circ_s\colon X\times X\to X$ for each $s\in S$ by the rule
	$$x\circ_s y=\tau_y(s)x.$$
	Then this axet is an $(f,\otimes)$-system.
\end{proposition}
\begin{proof}
	Let us check the axioms of $(f,\otimes)$-systems.
	
	1. Since $\tau_x(s)$ lies in the stabilizer of $x$, we have $x\circ_s x=\tau_x(s)x=x$.
	
	2. We need to check that the equation $x\circ_s y=z$ has a unique solution $x$ for all $y, z$. Indeed, we have $x=(\tau_y(s))^{-1}z$.
	
	3. The $\otimes$-selfdistributivity axiom is checked as follows. On one hand, $$(x\circ_s y)\circ_{s'} z = \tau_z(s')(\tau_y(s)x)=\tau_z(s')\tau_y(s)x.$$ On the other hand, due to the definition of $\otimes$, $$(x\circ_{s'} z)\circ_s(y\circ_{s'} z) = \tau_{\tau_z(s')y}(s)(\tau_z(s)x)=$$ $$=\tau_z(s')\tau_y(s)(\tau_z(s'))^{-1}\tau_z(s')x=\tau_z(s')\tau_y(s)x,$$ as requested.
\end{proof}

Now let $\im \tau$ be a subgroup in $G$. That means that for each $(x,s)\in X\times S$ there exists $(y,s')\in X\times S$ such that $\tau_y(s')=(\tau_x(s))^{-1}$. Then the following holds:

\begin{proposition}
	Define $\rho\colon X\times S\to X\times S$ by the rule $\rho(x,s)=(\rho_s(x),\rho_x(s))$ such that $$\tau_{\rho_s(x)}(\rho_x(s))=(\tau_x(s))^{-1}.$$ Then $\rho$ is a good involution on the associated quandle $(X\times S, \cdot)$.
\end{proposition}

\begin{proof}
	To simplify the notation we shall omit lower indices of $\rho$. Naturally, $\rho$ is an involution, so we need to check the two conditions of an involution being good. Let $u=(x,s), v=(y,s')$.
	
	1. $(u\cdot v)\cdot \rho(v)=((x\circ_{s'}y)\circ_{\rho(s')}\rho(y),s))=(\tau_{\rho(y)}(\rho(s'))\tau_y(s')x,s)=(x,s)=u$.
	
	2. Let us consider the left-hand and the right-hand sides of the equality we are checking. Left-hand side: $$\rho(u)\cdot v = (\rho(x)\circ_{s'}y,\rho(s)=(\tau_y(s')\rho(x),\rho(s)).$$ Right-hand side: $$\rho(u\cdot v)=\rho(x\circ_{s'}y,s)=(\rho(\tau_y(s')x),\rho(s))=\rho(\tau_y(s')x,s).$$ So we need to check whether it is true that $(\tau_{\tau_y(s')x}(s))^{-1}=\tau_{\tau_y(s')\rho(x)}(\rho(s))$. Indeed, the left-hand side of this equality is $$(\tau_{\tau_y(s')x}(s))^{-1}=\tau_y(s')(\tau_x(s))^{-1}(\tau_y(s'))^{-1},$$ and the right-hand side is $$\tau_{\tau_y(s')\rho(x)}(\rho(s))=\tau_y(s')\tau_{\rho(x)}(\rho(s))(\tau_y(s'))^{-1}=\tau_y(s')(\tau_x(s))^{-1}(\tau_y(s'))^{-1}$$ since $\tau_{\rho(x)}(\rho(s))=(\tau_x(s))^{-1}$ by definition of $\rho$.
	
	Therefore, $\rho$ is a good involution on $(X\times S,\cdot)$.
\end{proof}

To get a trivalent-compatible $(X,G,\{*_g\},f,\otimes,\oplus)$-system we need not just any good involution: we need to leave the first argument unchanged. It may easily be checked that if the axet is such that for any $x\in X$ the element $\tau_x(e)\in G$, where $e$ is the neutral element of $S$, acts trivially on $X$, then the mapping $\rho\colon (x,s)\mapsto (x,s^{-1})$ is a good involution. Then a simple check proves the following theorem:

\begin{theorem}
	Let $(X,G,\tau)$ be an $S$-axet for a commutative group $S$, and let $\tau_x(e)$ act trivially on $X$ for all $x\in X$. Define $f, \otimes$ as above, and let $$\rho\colon X\times S\to X\times S, \;\; \rho\colon (x,s)\mapsto (x,s^{-1}),$$ $$\oplus\colon S\times S\to S, \;\; \oplus\colon (s,s')\mapsto s's.$$ Then this axet is a trivalent-compatible $(X,G,\{*_g\},f,\otimes,\oplus)$-system with associative composition.
\end{theorem}
 
Note that this is just one example; in particular, there may exist other involutions which are good on the associated quandle of the system. 
%
%
%
%

\subsection{Fundamental associated quandle}

Just like in case of quandle colourings of knot diagrams, we can look at the invariant defined in the previous subsection in a different way. Namely, Let $D$ be an oriented diagram of a spatial trivalent graph (or a handlebody-link). We can define a (free) associated quandle $\mathcal{Q}_D = (X_D \times G_D, \cdot)$.  This quandle is generated by pairs $(x, g)$ for each arc of $D$ and is defined by the relations which one can write for all vertices and crossings as depicted in Fig~\ref{fig:color_rules_general}. Further, we factorize the quandle by relations 1--6 from Definition~\ref{def:tr-comp}. In case of handlebody-links we additionally factorize the quandle by the associativity condition from Definition~\ref{def:assoc}. As we see, these relations are the same as the ones which defined the proper colouring of a diagram. Let us call the resulting quandle the {\em fundamental (associated) quandle} of the graph (link).

In the same manner as for the proof of the colouring invariant theorem, we can prove that if a diagram $D'$ be obtained from $D$ by performing one of the moves on spatial trivalent graph diagrams (handlebody-link diagrams, respectively) with orientation which agrees with orientation of $D$ on the corresponding arcs, then the quandle $\mathcal{Q}_{D}$ goes to the isomorphic quandle $\mathcal{Q}_{D'}$, which corresponds to $D'$. Therefore, fundamental (associated) quandle is an invariant of spatial trivalent graphs (respectively, handlebody-links).

The colouring invariant defined in the previous section is then naturally interpreted as the number of quandle homomorphisms from the fundamental quandle $\mathcal{Q}_D$ into the given associated quandle $Q$.

%

 
\begin{question}
We have constructed an invariant of handlebody-links. If a handlebody-link does not have trivalent vertices, we get an invariant of classical links. What can be said on this invariant? In particular, since this invariant is a quandle, what is its relation with the fundamental quandle of the classical link?
\end{question}

\bigskip


\section{Invariants of spatial graphs and handlebody-links} \label{sec:graphs}

Consider spatial graphs in $\mathbb{R}^3$ up to ambient isotopy. They may be studied via their diagrams, which are  general position projections onto a plane $\mathbb{R}^2\subset \mathbb{R}^3$. It is known \cite{Kauf, Y} that two diagrams represent one and the same spatial graph if and only if they may be related by a sequence of moves shown in Fig.~\ref{fig:spat_moves} and their modifications obtained by changing all undercrossings to overcrossings and vice-versa. A representation of spatial graphs by braid and plane graphs can be found in \cite{BK}. 

\begin{remark}
As in case of knots, two spatial graphs are isotopic if and only if their diagrams are equivalent modulo a set of Reidemeister moves. In fact, there are two
slightly different notions of isotopy that are commonly considered:  {\em rigid (or flat) vertex isotopy}, in which the cyclic order of the edges at each vertex is fixed, and {\em pliable vertex isotopy}, or simply {\em isotopy}, in which the order of the edges around a vertex may change.
\end{remark}

\begin{figure}
\centering\includegraphics[width=350pt]{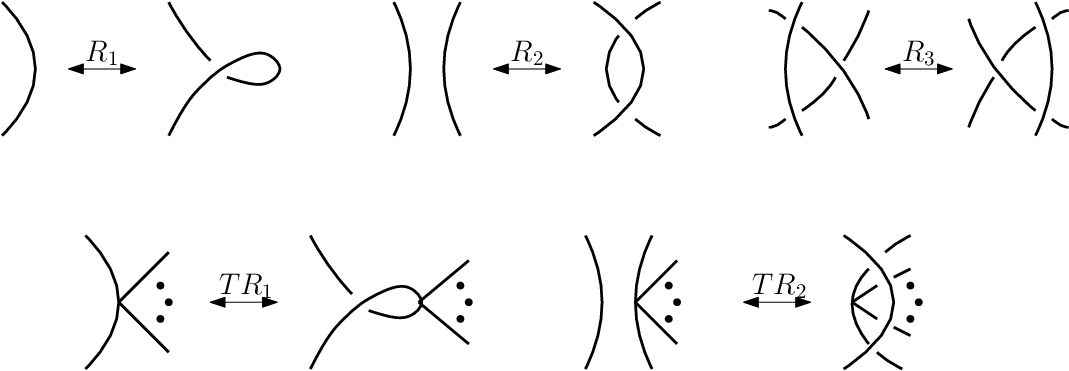}
\caption{Moves on spatial graph diagrams}\label{fig:spat_moves}
\end{figure}

In the following subsections we recall several simple invariants for spatial graphs (see, for example, \cite{M}) and handlebody-links. We also introduce a new invariant, based on $(f,\otimes)$-systems.

\subsection{Groups of spatial graphs} \label{ex:top_man_graph}

It is an interesting fact that we can not prove that man with linked fingers (MLF) is not equivalent to man with unlinked fingers (MUL) using the fundamental group of the complement  the spatial graph in $\mathbb{S}^3$. Indeed, the group of the spatial graph of the man with linked fingers (see Figure~\ref{fig:topol_man}a) is equal to
$$
H_1 = \langle a, b, c, d, f ~|~ac = c d,~~cd = d f,~~a^{-1} d b^{-1} = b f c^{-1} = 1 \rangle. 
$$
If we remove generators $d = a b$ and $c =  b f$ from last two relations, we get 
$$
H_1 = \langle a, b, f  ~|~[b f, a] b = 1 \rangle.
$$
Let us show that  this group is  isomorphic to free group of rank 2. To do it, we rewrite the defining relation in the form
$$
[g, a] b = 1,
$$
whee $g = bf$. Hence, we can remove $b$ and defining relation, and we get free group of rank 2.

On the other hand, the spatial graph corresponding to man with unlinked fingers is depicted in Figure~\ref{fig:topol_man}b. The group of this spatial graph  is isomorphic  to
$$
H_2 = \langle a, b, c ~|~a b^{-1} a^{-1}   =  c^{-1} b c  = 1 \rangle =  \langle a, c \rangle = F_2.
$$

Hence, the fundamental group of the complement  the spatial graph which corresponds to man with linked fingers is isomorphic to the complement  the spatial graph which corresponds to the man with unlinked fingers. 

\begin{proposition}
The fundamental group of the complement of a spatial graph in 3-sphere is isomorphic to the complement of the corresponding handlebody-link  in 3-sphere.
\end{proposition}

\begin{remark}
The spatial graph of MLF one can get by gluing an arc to Hopf link. The group of the Hopf link is isomorphic to the free abelian group $\mathbb{Z}^2$ of rank 2. On the other side, the spatial graph of MUF one can get by gluing an arc to trivial 2-component link.  The group of trivial 2-component link is isomorphic to the free  group $F_2$.
\end{remark}


\subsection{Handlebody-links and spatial graphs}
The theory of handlebody-links is different from the theory of spatial graphs. We present two example on that matter.

The first example was suggested by Kostya Storozhuk.
\begin{example}
Two handlebodies in Fig.~\ref{fig:athlete} (``Happy athlete'' and ``Unhappy athlete'') are equivalent, but they are not equivalent as spatial graphs. 

\begin{figure}
\centering\includegraphics[width=250pt]{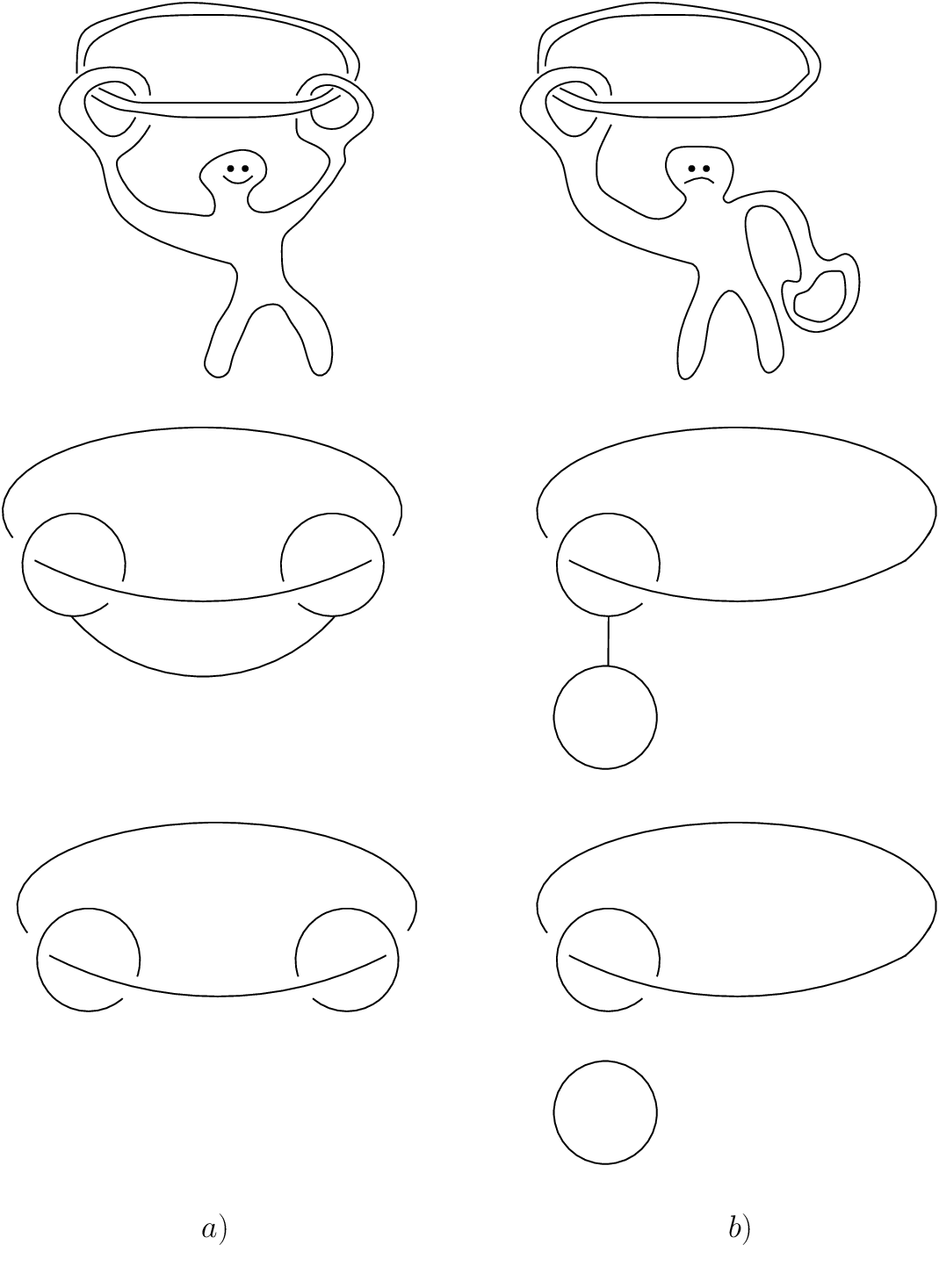}
\caption{``Happy athlete'' and ``unhappy athlete'' together with their graph representations and corresponding $L(\Gamma)$ invariants}\label{fig:athlete}
\end{figure}

\end{example}

The following invariant of spatial graphs was suggested by Kauffman in 1989 \cite{Kauf}. Let $\Gamma$ be a spatial graph. We can associate a collection $L(\Gamma)$ of knots and links by the rule: at each vertex choose two edges and delete the others. The set of all such links (with multiplicity for the same link may appear several times), for all choices of two edges at each vertex, is $L(\Gamma)$. It is an invariant of $\Gamma$. 

Let us prove using this invariant that as a spatial graph Athlete cannot unlink one hand. Indeed, Fig.~\ref{fig:athlete} shows graph-representations of the athletes (middle row) and the corresponding values of Kauffman invariant (bottom row; in both cases there is only one link in the collection $L(\Gamma)$. The links in the bottom row are evidently not equivalent, and therefore the initial graphs are not equivalent as well.

At the same time, these two handlebody-links are equivalent. To see that one can take the ``sole'' of one of the Happy athlete's fingers, and move it along its body, then the other pair of fingers, and back along the body. That will unlink the considered pair of fingers from the hoop it holds, making it unhappy indeed. That shows that the categories of handlebody-links and trivalent graphs are very different. It is an interesting exercise to present a sequence of moves on handlebody-link diagrams which constitute the described unlinking.

\begin{question}
Is it possible to modify the invariant $L(\Gamma)$ to get an invariant of handlebody-links?
\end{question}

%

Another example is related to the ``topological man'':

\begin{example}
The topological man with linked fingers  can unlink his fingers in the category of handlebody-links, but he cannot unlink fingers in the category of spatial graphs.
\end{example}

To prove it, we once again use the Kauffman invariant. Indeed, Fig.~\ref{fig:GRL} (top row) shows trivalent graphs of topological man (with watch) with a) linked fingers and b) unlinked fingers.  Bottom row depicts the corresponding collections $L(\Gamma)$ which consist of only one link in both cases. These links are evidently not equivalent. The same stays true for topological man without watch: in that case remove the circle component disjoint from the trivalent graph from all figures; the resulting links remain not equivalent.

\begin{figure} 
\centering\includegraphics[width=250pt]{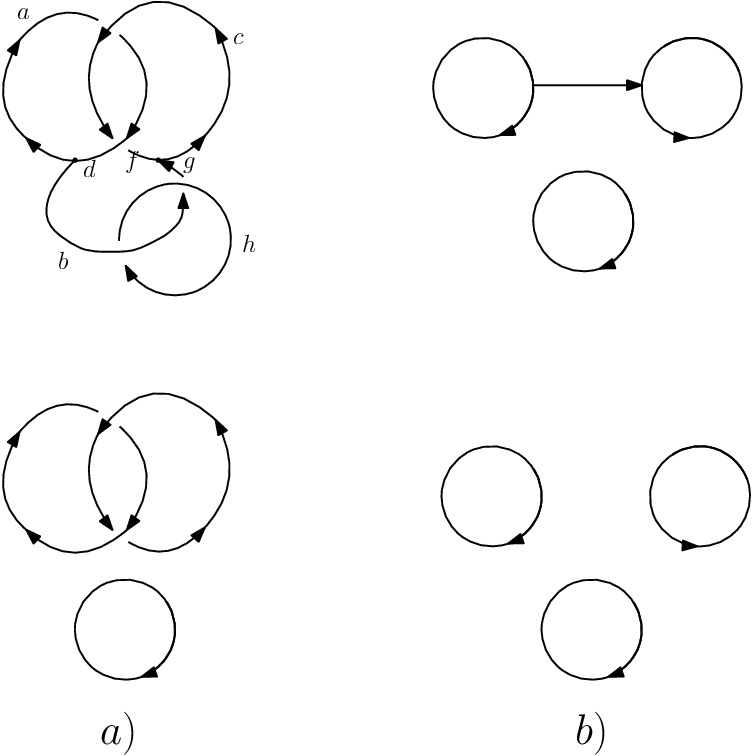} 
\caption{Spatial graphs corresponding to topological man (with watch) with a) linked fingers and b) unlinked fingers. Bottom row: corresponding links $L(\Gamma)$}\label{fig:GRL}
\end{figure}

\subsection{Topological man with watch}
In this subsection we are  considering topological man with watch and linked fingers (MWF) and prove that it is not equivalent to disjoint union of  unknotted handlebody of genus 2 and unknotted handlebody of genus 1 (topological man with watch and unlinked fingers, MWUF). 

At first, let us  find the fundamental group of the complement of the man with watch and linked fingers, $G_{WLF}$. Using 
Fig.~\ref{fig:GRL}(a, top row), we see that this group has a generating set
$$
a, b, c, d, f, g, h.
$$
In two vertices we have two relations
$$
d = a b,~~c = g f. 
$$

Further, consider relations in crossings. In the upper crossing (crossing 1) we have relation
$ac = cd.$ In crossing 2 we have relation $c d = d f.$ In crossing 3 we have relation $b h  = h g.$ In crossing 4 we have relation $h b = b f.$
Hence,
$$
G_{WLF} = \langle  a, b, c, d, f, g, h~|~ d = a b,~~c = g f,~~a c = c d,~~cd = d f,~~b f = h g,~~h b = b h \rangle.
$$
If we remove $d$ and $c$, using the first and the second relations, we get
$$
G_{WLF} = \langle  a, b,  f, g, h~|~ a  g f = g f a b,~~g f a b = a b  f,~~b h = h g,~~h b = b h \rangle.
$$
From the third relation, $g = h^{-1} b h$, and by the last, relation, $g = b$. Removing $g$, we get
$$
G_{WLF} = \langle  a, b,  f,  h~|~ a  b f = b f a b \rangle = \langle  h \rangle * \langle  a, b,  f~|~ a  b f = b f a b \rangle.
$$
We can rewrite the defining relation in the form $[bf, a] b = 1$. If we put $t = b f$, then
$$
\langle  h \rangle * \langle  a, b,  f~|~[bf, a] b = 1 \rangle = \langle  h \rangle * \langle  a, b,  t~|~ [t, a] b = 1 \rangle \cong F_3.
$$
We proved that $G_{WLF} \cong F_3$. On the other hand, it is easy to see that the fundamental group of the complement of the man with unlinked fingers and with watch is isomorphic to $F_3$ too.


Now let us construct quandle $Q_1$, which corresponds to MWF. Take a set and the group: 
$$
X_1 = \{ x, y, z\},~~G_1 = \langle  a, b, c, d, f, g, h \rangle
$$
Elements of $Q_1$  lie in the set  $X_1 \times G$ and this quandle is generated by elements:
$$
(x, a),~~(x, b),~~(y, c),~~(x, d),~~(y, f),~~(y, g),~~(z, h).
$$  
 In two vertices we have two relations
$$
d = a b,~~~ c = g f.
$$
 Further, consider relations in crossings. In the first crossing  we have relation
$$
(x, a) \cdot (y, c) = (x, d) \Leftrightarrow (x *_c y, c^{-1} a c) = (x, d),
$$
which is equivalent to the system
$$
\begin{cases}
x *_c y = x, & \\
c^{-1} a c = d. & 
\end{cases}
$$
In the second crossing we have relation
$$
(y, c) \cdot (x,d) = (y, f) \Leftrightarrow (y *_d x, d^{-1}  c d) = (y, f),
$$
which is equivalent to the system
$$
\begin{cases}
y *_d x = y, & \\
d^{-1}  c d = f. & 
\end{cases}
$$
In the third crossing we have relation
$$
(x, b) \cdot (z, h) = (y, g) \Leftrightarrow (x *_h z, h^{-1} b h) = (y, g),
$$
which is equivalent to the system
$$
\begin{cases}
x *_h z = y, & \\
h^{-1} b h = g. & 
\end{cases}
$$
In the forth crossing we have relation
$$
(z, h) \cdot (x, b) = (z, h) \Leftrightarrow (z *_b x, b^{-1} h b) = (z, h),
$$
which is equivalent to the system
$$
\begin{cases}
z *_b x = z, & \\
b^{-1} h b = h. & 
\end{cases}
$$

As was seen above,  $G_1 \cong F_3$ and we can present generators of $G_1$ as words in the generators $a$, $h$, $t$. We have
$$
b = g = [a, t],~~f = [t, a] t,~~d = a[a, t],~~c = t.
$$
Hence, we have proven

\begin{lemma} $Q_1$ is generated by elements
$$
(x, a),~~(x, [a, t]),~~(y, t),~~(x, a[a, t]),~~(y, [t, a] t),~~(y, [a, t]),~~(z, h),
$$
and is defined by relations
$$
x *_c y = x,~~y *_d x = y,~~x *_h z = y,~~z *_b x = z.
$$
\end{lemma}
 
\medskip 

Let us consider the disjoint union of  unknotted  handlebody of genus 2 and unknot  handlebody of genus 1 (MWUF). We construct a quandle $Q_2$, using Fig.~\ref{fig:GRL}(b, top row). We have 
$$
X_2 = \{ x, y \},~~G_2 = \langle  a, b, c,  h \rangle.
$$
Construct a quandle $Q_2$, which is generated by elements
$$
(x, a),~~(x, b),~~(x, c),~~(y, h) \in X_2 \times G_2.
$$  
In two vertices we have two relations
$$
 a  b = a,~~~b^{-1} c = c.
$$
Hence, $b=1$ and  as we seen before, $G_1 = G_2 = G_{WLF} \cong F_3$. We have 

\begin{lemma}
$Q_2$ is a subquandle of $X_2 \times G_2$, where $G_2 \cong F_3$, and it is generated by  elements
$$
(x, a),~~(x, 1),~~(x, c),~~(y, h).
$$
\end{lemma}

Let us find  colourings which we can construct using 3-element set $X = \{ a_0, a_1, a_2\}$ and cyclic group of order 2. At first, consider the epimorphism $G_2 \to \mathbb{Z}_2 = \langle \bar{h} \rangle$, which is defined by the rules
$$
a \mapsto 1,~~c \mapsto 1,~~h \mapsto \bar{h}.
$$ 
We can define $\mathbb{Z}_2$--family of quandles if we let $\R_3 = (X, *_1)$ to be the 3--element dihedral quandle, and let $\T_3 = (X, *_0)$ to be the 3-element trivial quandle. We have associated quandle $Q$ on the set $X \times \mathbb{Z}_2$. This quandle contains 6 elements. Define a colouring of MWUF by the rules
$$
(x, a) \mapsto (a_0, 1),~~(x, 1)\mapsto (a_0, 1),~~(x, c)\mapsto (a_0, 1),~~(y, h) \mapsto (a_1, \bar{h}).
$$
If we find a quandle, which is generated by these elements, we get 3-element quandle with multiplication
$$
\alpha \alpha = \alpha,~~\alpha \beta = \gamma,~~\alpha  \gamma = \alpha,
$$
$$
\beta \alpha = \beta,~~\beta \beta = \beta,~~\beta \gamma = \beta,
$$
$$
\gamma \alpha = \gamma,~~\gamma \beta = \alpha,~~\gamma \gamma = \gamma,
$$
where
$$
\alpha = (a_0, 1),~~ \beta =  (a_1, \bar{h}),~~\gamma = (a_2,  \bar{h}).
$$
Hence, MWUF has a colouring  by this 3-element quandle which is a subquandle of $Q$.

Analogously, if we consider the epimorphism $G_2 \to \mathbb{Z}_2 = \langle \bar{h} \rangle$, which is defined by the rules
$$
a \mapsto \bar{h},~~c \mapsto \bar{h},~~h \mapsto \bar{h},
$$ 
we can construct a colouring of MWUF by the associated quandle $Q$. Let us show that this colouring gives quandle $Q$. We have elements
$$
\delta_0 = (a_0, \bar{h}),~~\alpha_0 = (a_0, 1),~~\beta = (a_1, \bar{h}).
$$
Find products of these elements, we get
$$
\delta_0 \beta = (a_1, \bar{h}).
$$
Denote $(a_1, \bar{h})= \delta_2$. Further,
$$
\alpha_0 \delta_2 = (a_1, 1).
$$
Denote $f_1 = (a_1, 1)$.
Find the product
$$
\alpha_0 \beta = (a_2, 1).
$$
Denote $f_2 = (a_2, 1)$. We see that
$$
\{\alpha_0, ~~\alpha_1,~~\alpha_2,~~ \beta, ~~ \delta_0, ~~\delta_2 \} = R_1.
$$
Hence, we constructed a colouring of MWUF by quandle $Q$. To prove that $Q_1$ is not isomorphic to $Q_2$, it is enough to prove that there is no epimorphisms $Q_1 \to Q$
Define a homomorphism $Q_1 \to Q$, sending $X_1 \to \{a_0, a_1, a_2 \}$ by the rules
$$
x \mapsto a_0,~~~y \mapsto a_1,~~~z \mapsto a_2,
$$
and sending $G_1$ to $\mathbb{Z}_2$ be the rules
$$
t \mapsto \bar{h},~~a \mapsto \bar{h},~~h \mapsto \bar{h}.
$$ 
We have the following maps on other generators of $G_1$:
$$
b \mapsto 1,~~g \mapsto 1,~~f \mapsto \bar{h},~~d \mapsto \bar{h},~~c \mapsto \bar{h}.
$$ 
Hence, the generators of $Q_1$ go to the following elements of $Q$:
$$
(a_0, \bar{h}),~~(a_0, 1),~~(a_1, \bar{h}),~~(a_1, 1),~~(a_2, \bar{h}).
$$
Since, we want to construct a quandle homomorphism, relations of $Q_1$ must go to relations of $Q$. Hence, we have the following relations in $Q$:
$$
a_0 *_{\bar{h}} a_1 = a_0,~~a_1 *_{\bar{h}} a_0 = a_1,~~a_0 *_{\bar{h}} a_2 = a_1.
$$
But these are relations in the dihedral quandle $\R_3$. Hence, $a_0 = a_1 = a_2$ and the image of $Q_1$ is the trivial 2-element quandle
$$
(a_0, 1),~~(a_0, \bar{h}).
$$

In the general case, in $Q_1$ we have for relations
$$
x *_c y = x,~~y *_d x = y,~~z *_b x = z,~~x *_h z = y.
$$
To have these relations we must have in $G_1$:
$$
c = d = b = 1,
$$
or
$$
t = 1, ~~d = a =1,~~b = [a, t] = 1.
$$
It means that the homomorphism $G_1 \to  \mathbb{Z}_2$ is defined by the rules
$$
t \mapsto 1,~~a \mapsto 1,~~h \mapsto \bar{h}.
$$ 
In this case, all other generators $b, g, f, d, c$ go to the unit element of $ \mathbb{Z}_2$. Hence, the generator of $Q_1$ go to the following elements of $Q$:
$$
\alpha_0 = (a_0, 1),~~\alpha_1 = (a_1, 1),~~\beta_2 = (a_2, \bar{h}).
$$
Let us find a quandle which is generated by these elements. We have:
$$
\alpha_1 \alpha_0 =  \alpha_1,~~\beta_2 \alpha_0 = \beta_2,
$$
$$
\alpha_0 \alpha_1 =  \alpha_0,~~\beta_2 \alpha_1 = \beta_2,
$$
$$
\alpha_0 \beta_2 =  \alpha_1,~~\alpha_1 \beta_2 = \alpha_0.
$$
Hence, the image of $Q_1$ in $Q_2$ is a 3-element quandle. We have shown that the diagram of MWUF does not have a proper colouring by elements of $Q$.

\begin{theorem}
MWUF has a proper colouring by 6--elements quandle $Q = (\R_3 \times \mathbb{Z}_2, *_0, *_1)$, but MWF does not have such colouring. In particular, MWF is not equivalent to MWUF.
\end{theorem}

To provide alternative proof that MWF is not equivalent to MWUF we can use the construction proposed by A.~Malyutin. Consider the topological man with watch and unlinked fingers. In that case the watch (being an unlinked circle) can be contracted to a point. At the same time the watch on MWF can't be contracted. Therefore these two handlebody-links are not equivalent.

\subsection{The fundamental quandle of a spatial graph} \label{QuandleSpaGr} Niebrzydowski \cite{N} extended the notion of {\em fundamental quandle} (usually defined for links) to spatial graphs. Given a diagram for a spatial graph $\Gamma$, the fundamental quandle $Q(\Gamma)$ is the quandle whose generators are the arcs of the diagram, and apart from the usual quandle relations 1--3, there are additional Wirtinger-type relations given by each crossing and vertex. Using the arc labelling from in Figure~\ref{fig:spat_fund}, these additional relations are:

\begin{enumerate}
\item[4.] For every crossing we have the relation $a^b = c$, where we use the notation $a^b := a * b$ and $a^{\bar{b}} := a ~\bar{*}~ b$;
\item[5.] For each generator $x$ and every vertex we get the relation $x^{a_1^{\varepsilon_1} a_2^{\varepsilon_2} \ldots \,a_n^{\varepsilon_n}} = x$, where $a_i^{\varepsilon_i} = a_i$ if $\varepsilon_i = +1$, and $a_i^{\varepsilon_i} = \bar{a_i}$ if $\varepsilon_i = -1$; here $\varepsilon_i \in \{\pm 1\}$ and denote the orientation of the corresponding arc of the diagram.
\end{enumerate}
Relation~4 is the same one which comes from crossings in the definition of the fundamental quandle of a knot; relation~5 ensures invariance under Reidemeister moves for spatial graphs, so the quandle is an invariant of pliable vertex isotopy.

Given a quandle $Q$, there is an associated group $As(Q)$, obtained by interpreting the quandle operation as conjugation i.e. $x^y = y^{-1} x y$. For a knot, the associated group of the fundamental quandle is isomorphic to the fundamental group of the knot. However, for spatial graphs these groups are generally not isomorphic (although there is always an epimorphism from $As(Q)$ to the fundamental group of the graph). In particular, while the abelianization of the fundamental group of a knot is always an infinite cyclic group, the abelianization of $As(Q(\Gamma))$ is equal to $\mathbb{Z}^e$, where $e$ is the number of edges in the graph. So two cycles with different numbers of edges will be different, even if their fundamental groups are isomorphic.

\begin{figure}
\centering\includegraphics[width=250pt]{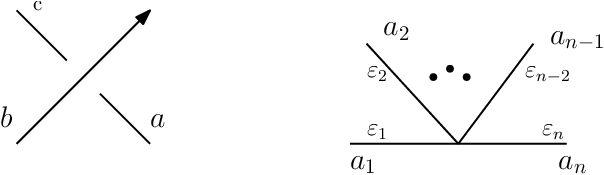}
\caption{Arc labelling for the fundamental quandle relations}\label{fig:spat_fund}
\end{figure}

\subsection{An invariant of spatial graphs coming from $(f,\otimes)$-systems}

Let us fix a natural number $n$. Consider an $(f,\otimes)$-system and let $\rho$ be an involution. Equip the system with a mapping $\Gamma\colon G^{\times n} \to G$.

\begin{definition}
	The $(X,G,\{*_g\},f,\otimes,\Gamma)$-system is called {\em $n$-compatible} if the involution $\rho$ on the associated quandle  is good, and the following conditions are satisfied:
	\begin{enumerate}
		\item $\Gamma$ is a composition function: for all $x,y\in X$, $g_1,\dots,g_n \in G$, $$((x\circ_{g_1}y)\circ_{g_2}\dots )\circ_{g_n} y = x\circ_{\Gamma(g_1,\dots,g_n)} y;$$
		\item for all $g_1,\dots,g_{n-2},h_1,h_2\in G$, $\Gamma(h_1,h_2,g_1,\dots,g_{n-2})=\Gamma(h_1,h_1 \otimes h_2, g_1,\dots, g_{n-2})$;
		\item the mapping $f$ depends only on its second argument: $f(g,h)=:f(h)$;
		\item for all $g_1,\dots,g_n,h\in G$, $h\otimes\Gamma(g_1,\dots,g_n)=((h\otimes g_1)\otimes\dots )\otimes g_n$ (twisted associativity);
		\item for all $g_1,\dots,g_n\in G$, $f(\Gamma(g_1,\dots,g_n))=\Gamma(f(g_1),\dots,f(g_n))$ 
		\item for all $g_1,\dots,g_n,h\in G$, $\Gamma(g_1,\dots,g_n)\otimes h=\Gamma(g_1\otimes h,\dots,g_n\otimes h)$ (distributivity);
		\item for all $x\in X, g_1,\dots,g_n\in G$, let $g_{n+1}=\Gamma(g_1,\dots,g_n)$. Then for any $i=0,\dots,n-1$, $$\rho_x(g_{n-i})=\Gamma(\rho_x(g_{n-i+1}),\dots,\rho_x(g_{n+1}),g_1,\dots,g_{n-i-1}).$$
	\end{enumerate}
\end{definition}

It is easy to see that for $n=2$ we get exactly the trivalent compatibility conditions.

We can also consider a set of operations $\Gamma_i$ for a set of distinct natural numbers $\Lambda=\{n_i\in\mathbb{N}, i=1,\dots,k\}$: $\Gamma_i\colon G^{\times n_i}_i\to G$. If $(f,\otimes,\Gamma_i)$-system is $n_i$-compatible for each $i$, we shall call $(f,\otimes,\Gamma_1,\dots,\Gamma_k)$-system {\em $\Lambda$-compatible}.

In parallel with Theorem~\ref{thm:tri-graph-inv} we get the following theorem, which directly follows from construction:

\begin{theorem} \label{thm:spa-graph-inv}
	Consider a set $\Lambda=\{n_1,\dots,n_k\}$ of distinct natural numbers, a $\Lambda$-compatible system with associated quandle $(\mathcal{A},\cdot)$ and a good involution $\rho$ which acts on the associated quandle in the way $\rho(x,g)=(x,\rho_x(g))$. Then the number of proper colourings with the associated quandle is invariant of spatial graphs with vertices of valencies $n_1,\dots,n_k$, where the colouring is {\em proper} if it satisfies the conditions depicted in Fig.~\ref{fig:proper_n_col}.
\end{theorem}

\begin{figure}
\centering\includegraphics[width=250pt]{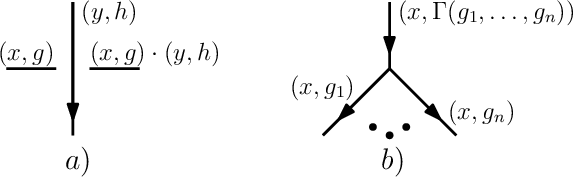}
\caption{Colouring rules for $n$-valent graph diagrams}\label{fig:proper_n_col}
\end{figure}

It turns out that trivalent-compatible $(X,G,\{*_g\},f,\otimes,\oplus)$-systems not only present an example of $\Lambda$-compatible systems for $\Lambda=\{2\}$, but also may be used to construct $\Lambda$-compatible systems for arbitrary $\Lambda$.

\begin{theorem}
	Consider a trivalent-compatible $(X,G,\{*_g\},f,\otimes,\oplus)$-system with associative composition function. Set $\Gamma(g_1,\dots,g_n)=((g_1\oplus g_2)\oplus\dots)\oplus g_n$. Then $(X,G,\{*_g\},f,\otimes,\oplus)$-system is $n$-compatible.
\end{theorem}

\begin{proof}
	We need to verify if the conditions of $n$-compatibility are satisfied. That can easily be checked directly and follows from the parallel conditions holding for $g\oplus h$ for all $g,h\in G$.
\end{proof}


\section{Open problems} \label{sec:questions}

\begin{question} 
Which invariants of spatial graphs can be transformed into invariants of handlebody-links? In particular, in subsection \ref{QuandleSpaGr} we provide a construction of fundamental quandle for spatial graphs. Is it possible to give a similar construction for handlebody-links?
\end{question} 

\begin{question}  
What solutions to the Yang-Baxter equations can you find using our quandle extensions? 
\end{question}

\begin{question}
What are some simple and useful examples of $(G,*,f)$-families of quandles? 
\end{question}

\begin{question}
 Niebrzydowski \cite{N} constructed fundamental quandle of spatial graph. If we take the quotient of this quandle by $SR$-move, we get invariant of handlebody-links. Is this   invariant strong?
 \end{question}


\begin{question}
What are the connections between our constructions with other constructions of quandle extensions? Quandle extensions were studied, for example, in \cite{A} and \cite{E}.
\end{question}

As we know that there exist certain functors from the category of groups to the category of quandles. For example, $\Conj$, $\Core$, $\Alex$ and some other.

\begin{question}
Are there  functors from category of $I$-multi-groups to the category of $I$-multi-quandles?
\end{question}

\begin{ack}
This work has been supported by the grant the Russian Science Foundation, RSF 24-21-00102, https://rscf.ru/project/24-21-00102/.

The authors would like to thank Andrei Malyutin, Sergey Melikhov, Kostya Storozhuk, and Matvey Zonov for useful comments, discussions, and suggestions.
\end{ack}

\end{document}